\documentclass[12pt]{amsart}
\usepackage{amsfonts,amssymb,amsmath,amscd,amsthm, amstext,bm,color}
\usepackage[colorlinks, citecolor=blue,pagebackref,hypertexnames=false]{hyperref}
\usepackage{enumerate}
\usepackage{epstopdf}
\usepackage{fancyhdr}
\usepackage{geometry}
\usepackage{mathrsfs}
\usepackage{txfonts}
\usepackage{tikz}
\usepackage{url}
\usepackage{ulem}
\usepackage[numbers,sort&compress]{natbib}
\allowdisplaybreaks

\newtheorem{theorem}{Theorem}[section]
\newtheorem{proposition}[theorem]{Proposition}
\newtheorem{lemma}[theorem]{Lemma}

\newtheorem{remark}[theorem]{Remark}

\newtheorem{definition}[theorem]{Definition}

\begin{document}

\newcommand{\norm}[1]{\left\Vert#1\right\Vert}
\newcommand{\abs}[1]{\left\vert#1\right\vert}
\newcommand{\set}[1]{\left\{#1\right\}}
\newcommand{\Real}{\mathbb{R}}
\newcommand{\RR}{\mathbb{R}^n}
\newcommand{\supp}{\operatorname{supp}}
\newcommand{\card}{\operatorname{card}}
\renewcommand{\L}{\mathcal{L}}
\renewcommand{\P}{\mathcal{P}}
\newcommand{\T}{\mathcal{T}}
\newcommand{\A}{\mathcal{A}}
\newcommand{\K}{\mathcal{K}}
\renewcommand{\SS}{\mathcal{S}}
\newcommand{\blue}[1]{\textcolor{blue}{#1}}
\newcommand{\red}[1]{\textcolor{red}{#1}}
\newcommand{\Id}{\operatorname{I}}

\title[CMO-Dirichlet  problem for the Schr\"odinger  equation]
{The CMO-Dirichlet  problem for the Schr\"odinger  equation in the upper half-space and characterizations of CMO}

\author[L. Song, L.C. Wu]{Liang Song  and    Liangchuan Wu*}
\thanks{*Corresponding author}

\address{
   Liang Song,
    Department of Mathematics,
    Sun Yat-sen University,
    Guangzhou, 510275,
    P.R.~China}
\email{songl@mail.sysu.edu.cn}

\address{
   Liangchuan Wu,
    Department of Mathematics,
    Sun Yat-sen University,
    Guangzhou, 510275,
    P.R.~China}
\email{wulchuan@mail2.sysu.edu.cn}


\begin{abstract}
 Let $\L$ be a Schr\"odinger operator of the form $\L=-\Delta+V$ acting on $L^2(\mathbb R^n)$ where the nonnegative
potential $V$ belongs to  the reverse H\"older class    ${\rm RH}_q$ for some $q\geq (n+1)/2$. Let ${\rm CMO}_{\L}(\mathbb{R}^n)$ denote the function space of vanishing mean oscillation associated to $\L$.
In this article we will show that a function $f$ of ${\rm CMO}_{\L}(\mathbb{R}^n) $ is the trace of  the solution to  $\mathbb{L}u=-u_{tt}+\L u=0$, $u(x,0)=f(x)$, if and only if, $u$ satisfies a Carleson condition
$$
\sup_{B: \ {\rm balls}}\mathcal{C}_{u,B} :=\sup_{B(x_B,r_B): \ {\rm balls}}  r_B^{-n}\int_0^{r_B}\int_{B(x_B, r_B)}   \big|t \nabla  u(x,t)\big|^2\, \frac{ dx\, dt } {t} <\infty,
$$
and
$$
     \lim _{a \rightarrow 0}\sup _{B: r_{B} \leq a}  \,\mathcal{C}_{u,B}
     =  \lim _{a \rightarrow \infty}\sup _{B: r_{B} \geq a}  \,\mathcal{C}_{u,B}
     =   \lim _{a \rightarrow \infty}\sup _{B: B \subseteq \left(B(0, a)\right)^c}  \,\mathcal{C}_{u,B}=0.
 $$
This  continues the lines of  the previous characterizations   by Duong, Yan and Zhang \cite{DYZ}  and Jiang and Li \cite{JL} for the ${\rm BMO}_{\L}$ spaces, which were founded by Fabes, Johnson  and Neri \cite{FJN} for the classical BMO space. For this purpose, we will prove two new characterizations of the ${\rm CMO}_{\L}(\mathbb{R}^n)$ space,  in terms of mean oscillation and the theory of tent spaces, respectively.
 \end{abstract}

\maketitle




\section{Introduction}\label{sec:1}
\setcounter{equation}{0}

The  space of  bounded mean oscillation  (BMO)   was introduced by John and Nirenberg \cite{JN}. A locally integrable function $f$ on ${\mathbb R}^n$ is said to be in  ${\rm BMO}({\mathbb R}^n)$,   if
$$
    \|f\|_{{\rm BMO}(\mathbb{R}^n)}=\sup_{B} \frac{1}{|B|}\int_B\left|f(y)-f_B\right| dy<\infty,
$$
where the supremum is taken over all balls $B\subseteq \mathbb{R}^n$ and   $f_B:=\frac{1}{|B|}\int_B f(x) \, dx$.

A celebrated theorem of Fefferman and Stein  \cite{FS} states that  a function $f$ of BMO is the trace of  the solution to
$$
 \begin{cases}
  \partial_{tt} u(x,t)+\Delta u(x,t)=0,     &(x,t)\in {\mathbb R}^{n+1}_+; \\
      u(x,0)=f(x),  \ & x\in \mathbb{R}^n
\end{cases}
$$
where  $u$ satisfies
\begin{equation}\label{eqn:FS-Carleson}
    \sup_{x_B, r_B} r_B^{-n}\int_0^{r_B}\!\int_{B(x_B, r_B)}   \big|t \nabla  u(x,t)\big|^2\, \frac{ dx\, dt } {t} <\infty,
\end{equation}
where $\nabla =(\nabla_x ,\,  \partial_t)$.  Expanding on this result, Fabes, Johnson and Neri \cite{FJN} showed that the condition \eqref{eqn:FS-Carleson} characterizes all the harmonic functions whose traces are in ${\rm BMO}(\mathbb{R}^n)$.   We refer the reader to \cite{SW1} for the earlier study of the  $H^p$ traces,  and to  \cite{ FN1, FN, DKP,JXY} for further results on  this topic.

In the last two decades,  the theory of BMO spaces associated to differential  operators  attracted lots of attentions. See for example, \cite{DGMTZ,ADM, DY1,AMR,HM}.  Especially, consider the Schr\"odinger operator
\begin{equation}\label{eqn:Schrodinger}
    \mathcal{L}=-\Delta+V(x) \   \    { \rm on } \  \  L^{2}(\mathbb{R}^{n}), \quad n \geq 3,
\end{equation}
where the nonnegative potential $V$ is not identically zero, and $V\in {\rm RH}_q$  for  some $q> n/2$. Recall that  $V\in {\rm RH}_q$  means that   $V\in L^{q}_{\rm loc}(\mathbb{R}^n), V\geq 0$, and
there exists a  constant $C>0$ such that  the reverse H\"older inequality
\begin{equation}\label{eqn:Reverse-Holder}
  \left(\frac{1}{|B|} \int_{B} V(y)^{q} d y\right)^{1 / q} \leq \frac{C}{|B|} \int_{B} V(y)\, d y
\end{equation}
holds for all   balls $B$ in $\mathbb{R}^n$. Recall that  $f$ belongs to  ${\rm BMO}_{{\mathcal{L}}}({\mathbb R}^n)$ (\cite{DGMTZ}) if  $f$ is a locally integrable function and satisfies
\begin{equation}\label{eqn:def-BMOL-norm}
     \|f\|_{{\rm BMO}_{\L}(\mathbb{R}^n)}:=\sup_{B=B(x_B,r_B):\, r_B<\rho(x_B)}\frac{1}{|B|} \int_{B}\left|f(y)-f_{B}\right|   d y+\sup_{B=B(x_B,r_B):\, r_B\geq \rho(x_B)} \frac{1}{|B|} \int_{B}|f(y)|\, d y <\infty.
\end{equation}
Here, the  function $\rho(x)$,  introduced by Z.W. Shen \cite{Shen,Shen1},  is defined by
\begin{equation}\label{eqn:critical-funct}
 \rho(x)=\sup \left\{r>0: \frac{1}{r^{n-2}} \int_{B(x, r)} V(y)\, d y \leq 1\right\}.
\end{equation}
Note that this ${\rm BMO}_{\L}(\mathbb{R}^n)$ space is a proper subspace of the classical $\rm BMO$ space and when $V\equiv 1$, ${\rm BMO}_{-\Delta+1}$ is just the ${\rm bmo}$ space  introduced by Goldberg \cite{Go1}. It is  known that  there is an alternative characterization of ${\rm BMO}_{\L}(\mathbb{R}^n)$ that $f\in {\rm BMO}_{\L}(\mathbb{R}^n)$, if and only if, $f\in L^2(\mathbb{R}^n, (1+|x|)^{-(n+\beta)}dx)$ for some $\beta>0$ and
\begin{equation}\label{eqn:def-BMOL-norm-II}
   \|f\|_{\widetilde{{\rm BMO}}_{\L}(\mathbb{R}^n)}:=\sup_B \left(\frac{1}{|B|} \int_B \left| f(x)-e^{-r_B\sqrt{ \L}}f(x) \right|^2 dx\right)^{1/2}<\infty,
\end{equation}
where the supremum is taken over all balls in $\mathbb{R}^n$. Moreover, $\|f\|_{{\rm BMO}_{\L}(\mathbb{R}^n)}\approx \|f\|_{\widetilde{{\rm BMO}}_{\L}(\mathbb{R}^n)}$.
See, for example, \cite[Proposition~6.11]{DY2}.

Recently, Duong, Yan and Zhang \cite{DYZ} extended the study  by Fabes {\it et al}  \cite{FJN} to the Dirichlet problem  for the Schr\"odinger equation with ${\rm BMO}_{\L}$ traces,  and established the following  characterization: whenever $V\in {\rm RH}_q$ with $q\geq n$, a solution $u$ to the  equation
\begin{align}\label{equations-L}
 \begin{cases}
  \partial_{tt} u(x,t)+\L u(x,t)=0,     &(x,t)\in {\mathbb R}^{n+1}_+; \\
      u(x,0)=f(x),  \ & x\in \mathbb{R}^n
     \end{cases}
 \end{align}
 satisfies \eqref{eqn:FS-Carleson}, if and only if, $u$ can be represented as $u=e^{-t\sqrt{\L}}f$, where $f$ is in $\rm BMO_{\L}(\mathbb{R}^n)$ space.
 Very recently, the condition  $V\in {\rm RH}_q$ with  $q\geq n$ in \cite{DYZ},   was improved  by Jiang and Li \cite{JL}   to  $q\geq (n+1)/2$.

On the other hand, it came to our attention that   Martell {\it et al} \cite{MMMM1} established the well-posedness of  the Dirichlet problem for any homogeneous, second-order, constant complex coefficient elliptic system  in the upper half-space,  with boundary data in the ${\rm VMO}(\mathbb{R}^n)$ space of Sarason  \cite{Sa2}. Here VMO is the  BMO-closure  of
 ${\rm UC} \cap {\rm BMO}$, where ${\rm UC}$ denotes the class of all uniformly continuous functions. There is yet another significant  space of functions of vanishing mean oscillations, ${\rm CMO}(\mathbb{R}^n)${\footnote{ CMO is also called VMO in \cite{CW}.}}, which is defined by the closure in the $\rm BMO$ norm of $C_c^\infty(\mathbb{R}^n)$, the space of smooth functions with compact support. Obviously ${\rm CMO}(\mathbb{R}^n)$ is a proper subspace of $\rm VMO(\mathbb{R}^n)$. One well-known fact  is that the Hardy space $H^1(\mathbb{R}^n)$ is the dual space  of ${\rm CMO}(\mathbb{R}^n)$; see \cite[Theorem~4.1]{CW}.  Notably, in a subsequent  paper \cite{MMMM2} by Martell {\it et al}, the authors posed an open question to formulate and prove a well-posedness result for the Dirichlet problem in the upper half-space, for the elliptic system as in \cite{MMMM1} with $\rm CMO$ traces.

\smallskip

The main aim of this article is to  continue the lines of \cite{DYZ, JL} to process the study on the Dirichlet problem for the Schr\"odinger equation with boundary value in ${\rm CMO}_{\L}(\mathbb{R}^n)$.  The  ${\rm CMO}_{\L}(\mathbb{R}^n)$ is the space of functions of vanishing mean oscillation associated to $\L$, introduced  in \cite{DDSTY} by Deng, Duong, Tan, Yan and the first named author of this article  under a more general  setting.

\begin{definition}[\cite{DDSTY}]\label{def:CMO-Schrodinger-bySong}
  We say that a function $f\in {\rm CMO}_{\L}(\mathbb{R}^n)$ if  $f$ is in $ {\rm BMO}_{\L}(\mathbb{R}^n)$ and  satisfies the limiting conditions $\gamma_1(f)=\gamma_2(f)=\gamma_3(f)=0$, where
  \begin{align*}
    \gamma_{1}(f)&= \lim _{a \rightarrow 0}\sup _{B: r_{B} \leq a}\left(r_B^{-n}\int_{B}\left|f(x)-e^{-r_B\sqrt{\L}} f(x)\right|^{2} d x\right)^{1 / 2} ;\\
    \gamma_{2}(f)&= \lim _{a \rightarrow \infty}\sup _{B: r_{B} \geq a}\left(r_B^{-n} \int_{B}\left|f(x)-e^{-r_B\sqrt{\L}} f(x)\right|^{2} d x\right)^{1 / 2} ;\\
    \gamma_{3}(f)&= \lim _{a \rightarrow \infty}\sup _{B \subseteq \left(B(0, a)\right)^c}\left(r_B^{-n} \int_{B}\left|f(x)-e^{-r_B\sqrt{\L}} f(x)\right|^{2} d x\right)^{1 / 2}.
\end{align*}
We endow ${\rm CMO}_{\L}(\mathbb{R}^n)$ with the norm of ${\rm BMO}_{\L}(\mathbb{R}^n)$.
\end{definition}

Note that, whenever $\L=-\Delta$, i.e., $V\equiv 0$, the space ${\rm CMO}_{-\Delta}(\mathbb{R}^n)$   coincides  with ${\rm CMO}(\mathbb{R}^n)$; see  \cite[Proposition~3.6]{DDSTY}. Besides, the following results hold.

\begin{theorem}\label{thm:CMO-dual-known}
Suppose $V\in {\rm RH}_q$ for some $q> n/2$.
\begin{itemize}
	
	\item[(i) ]{\rm (\cite{DDSTY})}  ${\rm CMO}_{\L}(\mathbb{R}^n)$  is the predual space of $H_{\L}^1(\mathbb{R}^n)$.

	\smallskip

\item[(ii) ] {\rm (\cite{Ky})} ${\rm CMO}_{\L}(\mathbb{R}^n)$ is the closure of $C_c^\infty(\mathbb{R}^n)$ in the ${\rm BMO}_{\L}(\mathbb{R}^n)$ norm.
\end{itemize}
\end{theorem}

\smallskip

The reader is referred to  Section~2 for  the definition of $H_{\L}^1(\mathbb{R}^n)$, the Hardy space associated to $\L$.
We say that $u\in W^{1,2}(\mathbb{R}_+^{n+1})$ is an $\mathbb{L}$-harmonic function in $\mathbb{R}_+^{n+1}$, if $u$ is a weak solution of $\mathbb{L}u:=-u_{tt}+\L u=0$, that is,
$$
    \iint_{\mathbb{R}_+^{n+1}} \partial_{t} u \partial_{t} \phi \, dx\,dt+    \iint_{\mathbb{R}_+^{n+1}}  \left\langle\nabla_{x} u, \nabla_{x} \phi\right\rangle \, dx\,dt+\iint_{\mathbb{R}_+^{n+1}}  V u \phi \, dx\,dt=0
$$
holds for all Lipschitz functions $\phi$ with compact support in $\mathbb{R}_+^{n+1}$. The space ${\rm HMO}_{\L}(\mathbb{R}_+^{n+1})$ is defined as the class of all $\mathbb{L}$-harmonic functions $u$, that satisfies
$$
   \|u\|_{{\rm HMO}_{\L}(\mathbb{R}_+^{n+1})}:=\sup_{x_B,r_B} \left( r_{B}^{-n}\int_0^{r_B} \int_{B(x_B,r_B)} |t\nabla u(x,t)|^2 \frac{dx\, dt}{t}\right)^{1/2}<\infty,
$$
where $\nabla :=(\partial_t,\nabla_x)$.

\begin{definition}
We say  that $u$ belongs to  ${\rm HCMO}_{\L}(\mathbb{R}_+^{n+1})$ if   $u\in {\rm HMO}_{\L}(\mathbb{R}_+^{n+1})$, and  satisfies the limiting conditions $\beta_1(u)=\beta_2(u)=\beta_3(u)=0$, where
\begin{align*}
    \beta_{1}(u)&= \lim _{a \rightarrow 0}\sup _{B:\, r_{B} \leq a}\left(r_B^{-n} \int_0^{r_B}\int_{B}\left|t\nabla u(x,t)\right|^{2} \frac{dx\, dt}{t}\right)^{1 / 2} ;\\
    \beta_{2}(u)&= \lim _{a \rightarrow \infty}\sup _{B: \, r_{B} \geq a}\left(r_B^{-n} \int_0^{r_B}\int_{B}\left|t\nabla u(x,t)\right|^{2} \frac{dx\, dt}{t}\right)^{1 / 2};\\
    \beta_{3}(u)&= \lim _{a \rightarrow \infty}\sup _{B: \,  B \subseteq \left(B(0, a)
    \right)^c}\left(r_B^{-n} \int_0^{r_B}\int_{B}\left|t\nabla u(x,t)\right|^{2} \frac{dx\, dt}{t}\right)^{1 / 2},
\end{align*}
We endow ${\rm HCMO}_{\L}(\mathbb{R}_+^{n+1})$ with the norm of ${\rm HMO}_{\L}(\mathbb{R}_+^{n+1})$.
\end{definition}
\smallskip

The main result of this paper  is  the following characterization.

\smallskip

\noindent{\bf Theorem A.} {\it \
Suppose $V\in {\rm RH}_q$ for some $q\geq (n+1)/2$ and  let $\L=-\Delta+V$.

\begin{itemize}
\item [(i) ] If $u\in {\rm HCMO}_{\L}(\mathbb{R}_+^{n+1})$, then there exists a function $f\in {\rm CMO}_{\L}(\mathbb{R}^n)$ such that $u(x,t)=e^{-t\sqrt{\L}}f(x)$, and  there exists a constant $C>1$, independent of $u$, such that
$$
    \|f\|_{{\rm BMO}_{\L}(\mathbb{R}^n)}\leq C \|u\|_{{\rm HMO}_{\L}(\mathbb{R}_+^{n+1})}.
$$

\smallskip

\item[(ii) ] If $f\in {\rm CMO}_{\L}(\mathbb{R}^n)$, then $u(x,t)=e^{-t\sqrt{\L}}f(x)\in {\rm HCMO}_{\L}(\mathbb{R}_+^{n+1})$, and there exists a constant $C>1$, independent of $f$, such that
$$
   \|u\|_{{\rm HMO}_{\L}(\mathbb{R}_+^{n+1})}\leq C\|f\|_{{\rm BMO}_{\L}(\mathbb{R}^n)}.
$$
\end{itemize}
}

\smallskip

Based on the previous works \cite{DYZ, JL} on this Dirichlet problem with ${\rm BMO}_{\L}(\mathbb{R}^n)$ traces, the  main difficulty of proving Theorem~A is to reveal the connections between limiting conditions equipped by solutions and traces, respectively.
In order to show (i) of Theorem~A, we will establish an equivalent characterization of the space ${\rm CMO}_{\L}(\mathbb{R}^n)$ in terms of  tent spaces.

Let $T_2^p$, $0<p\leq \infty$, be the classical tent spaces introduced by Coifman {\it et al}  in \cite{CMS1, CMS2} (see Section 2  for precise definitions). Let $T_{2, c}^{2}$ denote the set of all $f \in T_{2}^{2}$ with compact support in $\mathbb{R}_{+}^{n+1}$. Denote by	$T_{2,C}^{\infty}$ the closure of the set $T_{2, c}^{2}$ in $T_{2}^{\infty}$, and we endow $T_{2,C}^{\infty}$ with the norm of $T_{2}^{\infty}$. The following result is a special case of Proposition~3.3 in \cite{DDSTY}, by taking the operator therein to be the Schr\"odinger operator $\L$.

\begin{proposition} \label{prop:CMO-Tent}
 Suppose $V\in {\rm RH}_q$ for some $q> n/2$  and  let $\L=-\Delta+V$.
Then $f\in {\rm CMO}_{\L}$ if and only if $f\in L^2(\mathbb{R}^n, (1+|x|)^{-(n+\beta)}dx)$ for some $\beta>0$ and $t\sqrt{\L}e^{-t\sqrt{\L}} (I-e^{-t\sqrt{\L}}) f\in T_{2,C}^\infty$, with
$$
    \|f\|_{{\rm CMO}_{\L}}\approx \left\|  t\sqrt{\L}e^{-t\sqrt{\L}}\left(I-e^{-t\sqrt{\L}}\right) f \right\|_{T_2^\infty}.
$$
\end{proposition}

However,  the above proposition can not be used directly to  show (i) of Theorem~A. As a result, we  have to establish  a revised version of Proposition~\ref{prop:CMO-Tent}, Theorem B below.

\smallskip

\noindent{\bf Theorem B.} {\it  \
Suppose $V\in {\rm RH}_q$ for some $q> n/2$.
Then $f\in {\rm CMO}_{\L}$ if and only if  $f\in L^2(\mathbb{R}^n, (1+|x|)^{-(n+\beta)}dx)$ for some $\beta>0$ and $t\sqrt{\L}e^{-t\sqrt{\L}} f\in T_{2,C}^\infty$, with
$$
    \|f\|_{{\rm CMO}_{\L}}\approx \left\|  t\sqrt{\L}e^{-t\sqrt{\L}}f \right\|_{T_2^\infty}.
$$
}

\smallskip

We note  that analogous versions of Theorem~B have been established for second order divergence form elliptic operators  in \cite{SX,JY1}.
  The argument in the proof of Theorem B is  based on a modification of techniques in  \cite{HM, SX}.

\smallskip

On the other hand, to prove (ii) of Theorem~A, we will give another new characterization of ${\rm CMO}_{\L}(\mathbb{R}^n)$ in terms of limiting behaviors of mean oscillation.

\medskip

\noindent{\bf Theorem C.} {\it \
 Suppose $V\in {\rm RH}_q$ for some $q> n/2$  and  let $\L=-\Delta+V$. The following statements are equivalent.

\begin{itemize}
\item [(a) ] $f$ is in ${\rm CMO}_{\L}(\mathbb{R}^n)$.

\medskip

\item [(b) ] $f$  is in the closure of $C_0(\mathbb{R}^n)$ in ${\rm BMO}_{\L}(\mathbb{R}^n)$, where $C_0(\mathbb{R}^n)$ is the space of all continuous functions on $\mathbb{R}^n$ which vanish at infinity.

\medskip

\item [(c) ] $f$ is in $\mathcal{B}_{\L}$, where $\mathcal{B}_{\L}$ is the subspace of ${\rm BMO}_{\L}(\mathbb{R}^n)$ satisfying $\widetilde{\gamma}_i(f)=0$ for $1\leq i\leq 5$, where
\begin{align*}
\widetilde{\gamma}_1(f) &=\lim _{a \rightarrow 0} \sup _{B: \,r_{B} \leq a}\left(|B|^{-1} \int_{B}\left|f(x)-f_{B}\right|^{2} d x\right)^{1 / 2} ;\\
\widetilde{\gamma}_2(f) &=\lim _{a \rightarrow \infty} \sup _{B: \,r_{B} \geq a}\left(|B|^{-1} \int_{B}\left|f(x) -f_B\right|^{2} d x\right)^{1 / 2} ;\\
\widetilde{\gamma}_3(f) &=  \lim _{a \rightarrow \infty} \sup _{B: \, B \subseteq (B(0, a))^c}\left(|B|^{-1} \int_{B}\left|f(x)-f_{B}\right|^{2} d x\right)^{1 / 2};\\
\widetilde{\gamma}_4(f) &=\lim _{a \rightarrow \infty} \sup _{B: \,r_{B} \geq \max\{a,\,\rho(x_B)\}}\left(|B|^{-1} \int_{B}\left|f(x) \right|^{2} d x\right)^{1 / 2} ;\\
\widetilde{\gamma}_5(f) &=  \lim _{a \rightarrow \infty} \sup _{\substack{B: \, B \subseteq (B(0, a))^{c}\\   r_B\geq \rho(x_B)}}\left(|B|^{-1} \int_{B}\left|f(x)\right|^{2} d x\right)^{1 / 2}.
\end{align*}
Here $x_B$ denotes the center of $B$,  and  the function $\rho$ is defined in \eqref{eqn:critical-funct}.

\medskip

\item [(d)] $f$ is in ${\rm BMO}_{\L}(\mathbb{R}^n)$  and satisfies $\widetilde{\gamma}_1(f)=\widetilde{\gamma}_3(f)=\widetilde{\gamma}_5(f)=0$.
\end{itemize}
}

\medskip

Recall that Uchiyama \cite{U} proved that $f\in{\rm CMO}(\mathbb{R}^n)$ if and only if $f\in \mathcal{B}$, where  $\mathcal{B}$ is the subspace of ${\rm BMO}(\mathbb{R}^n)$ satisfying
 \begin{subequations}
\begin{equation}\label{eqn:U-CMO-a}
     \lim _{a \rightarrow 0} \sup _{B: \, r_{B} \leq a}  \frac{1}{|B|} \int_{B}\left|f(x)-f_{B}\right| d x=0;
\end{equation}
\begin{equation}\label{eqn:U-CMO-b}
      \lim _{a \rightarrow \infty} \sup _{B: \, r_{B} \geq a}\frac{1}{|B|} \int_{B}\left|f(x)-f_{B}\right| d x=0;
\end{equation}
\begin{equation}\label{eqn:U-CMO-c}
     \lim _{a \rightarrow \infty} \sup _{B: \, B \subseteq (B(0, a))^c}\frac{1}{|B|} \int_{B}\left|f(x)-f_{B}\right| d x=0.
\end{equation}
\end{subequations}
It should be pointed out that the above result was first announced by Neri \cite{Ne} without proof and  the three limiting conditions above  are  mutually independent (see  \cite[p. 49]{D} for some examples).
   Theorem~C may be seen as a generalization of the  Neri--Uchiyama theorem  from the  ${\rm CMO}_{-\Delta}$ space to the ${\rm CMO}_{\L}$ space. Indeed, when $V\equiv 0$,  the auxiliary function $\rho(x)\equiv \infty$ for each $x\in \mathbb{R}^n$,  then the last two requirements in (c) of Theorem C,  $\widetilde{\gamma}_4(f)=0$ and $\widetilde{\gamma}_5(f)=0$,  are trivial.

In the case of  $V\equiv 1$,  it is well known that ${\rm BMO}_{-\Delta+1}={\rm bmo}$,  which is  the dual of the local Hardy space $h^1$(\cite{Go1}).  The pre-dual space of $h^1$ is a local version of CMO, which can also be regarded as our  ${\rm CMO}_{-\Delta+1}(\mathbb{R}^n)$. It was proved by Dafni in \cite{D} that $f\in {\rm CMO}_{-\Delta+1}$ is equivalent to $f\in {\rm BMO}_{-\Delta+1}$ and $\widetilde{\gamma}_1(f)=\widetilde{\gamma}_5(f)=0$. However, for general $V\in {\rm RH}_q$, $q> n/2$, the situation is different. Theorem C states that $f\in {\rm CMO}_{\L}$ is equivalent to $f\in {\rm BMO}_{\L}$ and $\widetilde{\gamma}_1(f)=\widetilde{\gamma}_3(f)=\widetilde{\gamma}_5(f)=0$. Also,  We will construct an example at the end of Section 4, which satisfies $\widetilde{\gamma}_1(f)=\widetilde{\gamma}_5(f)=0$, while $\widetilde{\gamma}_3(f)\neq 0$. This implies that  $\widetilde{\gamma}_3(f)=0$ can not be deduced by $\widetilde{\gamma}_1(f)=\widetilde{\gamma}_5(f)=0$.

The main difficulty of showing Theorem~C  arises from the implicit  function $\rho(x)$ occurring in  $\widetilde{\gamma}_4(f)$ and $\widetilde{\gamma}_5(f)$. Concretely, even though it is known that $\rho(x)$ is a slowly varying function (see Lemma~\ref{lem:size-rho}),  there is no a uniformly positive bound  (from above or below) for $\rho(x)$.  For this reason,  to verify  the averaged  behaviors of functions on balls in the case of $r_B\geq \rho(x_B)$  becomes more subtle.
  For clarity, we will begin by  showing  the standard modifier is not sufficient to approximate a  given function in $\mathcal{B}_\L$ directly (see \eqref{eqn:BMO-approx-b} in Lemma~\ref{lem:CMO-identity-approx}), although such an approach   has been successfully applied in \cite[Theorem~6]{D} to character the   ${\rm CMO}_{-\Delta+1}$ space.
The difficulty will be overcome in this article by combining a  modified Uchiyama's  construction  (see Lemma~\ref{lem:approx-type-II}) and the standard modifier, relied heavily on  properties of $\rho(x)$.

\smallskip

The layout of the article is as follows. In Section~2, we recall some preliminary results, including the theory of tent spaces and the kernel estimates of the heat and Poisson semigroups of $\L$.
Section~3 is mainly devoted to show Theorem~B, by combining  ideas of \cite{HM, DDSTY, SX}.
 Our  purpose  in Section~4  is to prove Theorem~C,  based on two auxiliary estimates following from the standard approximation to the identity and  a modified Uchiyama's construction, respectively. In Section 5,   Theorem~A is proved by applying Theorem ~B and  Theorem~C.

Throughout this article, the letter ``$C$'' or ``$c$'' will denote (possibly different) constants that are independent of the essential variables.  By $A\approx B$ (resp. $A\lesssim B$), we mean that there exists a positive constant $C$ such that $C^{-1}A\leq B\leq CA$ (resp. $A\leq CB$).

\bigskip




\section{Preliminaries and auxiliary results associated to operators}\label{sec:2}
\setcounter{equation}{0}

In this section, we recall some basic definitions and properties of  tent spaces and the critical radii function $\rho(x)$.

\smallskip

\subsection{Tent Spaces}

Let $\Gamma(x)=\{(y,t)\in \mathbb{R}_+^{n+1}:\, |x-y|<t \}$ be the standard cone (of aperture 1) with vertex $x\in \mathbb{R}^n$. For any closed subset $F\subseteq \mathbb{R}^n$, $\mathcal{R}(F):=\bigcup_{x\in F}\Gamma(x)$. If $O$ is an open subset of $\mathbb{R}^n$, then the ``tent'' over $O$, denoted by $\widehat{O}$, is given as $\widehat{O}=[\mathcal{R}(O^c)]^c$.

For any function $F(y,t)$ defined on $\mathbb{R}_+^{n+1}$, we will denote
$$
    \mathcal{A}(F)(x)=\left(\iint_{\Gamma(x)} |F(y,t)|^2 \frac{dy\, dt}{t^{n+1}}\right)^{1/2}
$$
and
$$
    \mathcal{C}(F)(x)=\sup_{x\in B} \left(r_B^{-n} \iint_{\widehat{B}} |F(y,t)|^2 \frac{dy\, dt}{t} \right)^{1/2}.
$$

As in \cite{CMS2}, the tent space $T_2^p$ is defined as the space of functions $F$ such that $\mathcal{A}(F)\in L^p(\mathbb{R}^n)$ when $p<\infty$. The resulting equivalence classes are then equipped with the norm $\|F\|_{T_{2}^p}=\|\mathcal{A}(F)\|_p$. When $p=\infty$, the space $T_2^\infty$ is the class of functions $F$ for which $\mathcal{C}(F)\in L^\infty(\mathbb{R}^n)$ and the norm $\|F\|_{T_2^\infty}=\|\mathcal{C}(F)\|_\infty$. Let $T_{2, c}^{p}$ be the set of all $f \in T_{2}^{p}$ with compact support in $\mathbb{R}_{+}^{n+1}$. We denote by $T_{2,C}^{\infty}$ the closure of the set $T_{2, c}^{2}$ in $T_{2}^{\infty}$, and we endow $T_{2,C}^{\infty}$ with the norm of $T_{2}^{\infty}$.

Let $\mathcal{H}$ be the set of all $f\in T_2^\infty$ satisfying the following three conditions:
\begin{itemize}
	\item [(i)] $\displaystyle
    \eta_1(F):=\lim_{a\to 0}   \sup_{B:\, r_B\leq a} \left(r_B^{-n} \iint_{\widehat{B}} |F(y,t)|^2  \frac{dy\, dt}{t}\right)^{1/2}=0$;
	\smallskip
	\item [(ii)] $\displaystyle \eta_2(F):=\lim_{a\to +\infty}  \sup_{B:\, r_B\geq a} \left(r_B^{-n} \iint_{\widehat{B}} |F(y,t)|^2  \frac{dy\, dt}{t}\right)^{1/2}=0$;
	\smallskip
	\item [(iii)] $\displaystyle \eta_3(F):=\lim_{a\to +\infty}  \sup_{B:\, B\subseteq \left(B(0,a)\right)^c} \left(r_B^{-n} \iint_{\widehat{B}} |F(y,t)|^2  \frac{dy\, dt}{t}\right)^{1/2}=0$.
\end{itemize}
It can be verified that $\mathcal{H}$ is a closed linear subspace of $T_{2}^{\infty}$.

\begin{lemma}\label{lem:predual-tent}
 (a)  $ (T_{2,C}^{\infty})^{*}=T_2^1$, i.e., $T_2^1$ is the dual space of $T_{2,C}^\infty$.

  \smallskip
 (b)  $f \in T_{2,C}^{\infty}$ if and only if $f \in \mathcal{H}$.
\end{lemma}
\begin{proof}
 (a) was proved  in \cite[Theorem~1.7]{W}.  (b) was proved in \cite[Lemma~3.2]{DDSTY}.
\end{proof}

\smallskip

\subsection{Basic properties of the critical radii function $\rho(x)$ }
In this subsection, we recall some basic properties of the critical radii function $\rho(x)$ defined in \eqref{eqn:critical-funct}, which were  first proved  by Z.W. Shen in \cite{Shen1}.

\begin{lemma}\label{V-integral size} {\rm (\cite[Lemma~1.2]{Shen1}.)}  Suppose $V\in {\rm RH}_q$ for $q>1$. There exists $C>0$ such that, for $0<r<R<\infty$,
$$
\frac{1}{r^{n-2}} \int_{B(x,r)} V(y) \, dy\leq C \left(\frac{R}{r}\right)^{\frac{n}{q}-2} \frac{1}{R^{n-2}} \int_{B(x,R)} V(y) \, dy.
$$
\end{lemma}

\smallskip

\begin{lemma} \label{lem:size-rho}
{\rm ( \cite[Lemma~1.4]{Shen1}.)}\
Suppose $V\in {\rm RH}_q$ for some $q> n/2$. There exist $c>1$ and $k_0\geq1$ such that for all $x,y\in\mathbb{R}^n$,
\begin{equation}\label{eqn:size-rho}
c^{-1}\left(1+\frac{\abs{x-y}}{\rho(x)}\right)^{-k_0} \rho(x)\leq\rho(y)\leq c
\left(1+\frac{\abs{x-y}}{\rho(x)}\right)^{\frac{k_0}{k_0+1}} \rho(x).
\end{equation}
In particular, $\rho(x)\approx \rho(y)$ when $y\in B(x, r)$ and $r\lesssim\rho(x)$.
\end{lemma}

Noting that $\rho(x)>0$ for each $x\in \mathbb{R}^n$, Lemma~\ref{lem:size-rho} implies that the implicit function $\rho$ is locally bounded from above and below.   This fact will be used frequently in the sequel.

\smallskip

\subsection{Basic properties of the heat and Poisson semigroups of Schr\"odinger operators}

Let $\left\{e^{-t\L}\right\}_{t>0}$ be the heat  semigroup associated to $\L$:
\begin{equation}\label{eqn:heat-semigroup}
 e^{-t\L}f(x)=\int_{\mathbb{R}^n}{\mathcal K}_t(x,y)f(y)~dy,\qquad f\in L^2(\mathbb{R}^n),\  x\in\mathbb{R}^n,\   t>0.
\end{equation}

\begin{lemma}\label{lem:heat-Schrodinger}
{\rm (see \cite{DGMTZ}.)} ~~~
Suppose $V\in {\rm RH}_q$ for some $q> n/2$. For every $N>0$ there exist constants $C_N$ and $c$ such that for $ x,y\in\mathbb{R}^n, t >0$,
\smallskip
 \begin{itemize}
\item[(i)]
$\displaystyle  0\leq {\mathcal K}_t(x,y)\leq C_N\,t^{-n/2}\exp\left(-\frac{|x-y|^2}{ct}\right)\left(1+\frac{\sqrt{t}}{\rho(x)} +\frac{\sqrt{t}}{\rho(y)}\right)^{-N}  $;

\smallskip

\item[(ii)]
$\displaystyle  \left|\partial_t{\mathcal K}_t(x,y)\right|\leq C_N \, t^{-\frac{n+2}{2}}\exp\left(- \frac{|x-y|^2}{ct}\right)  \left(1+\frac{\sqrt t}{\rho(x)}+\frac{\sqrt t}{\rho(y)}\right)^{-N} $.
 \end{itemize}
\end{lemma}

Denote by $h_t(x)$
 the kernel of the classical heat semigroup $\big\{e^{t\Delta}\big\}_{t>0}$ on $\mathbb{R}^n$. Then  the following result is valid.

\begin{lemma} \label{lem:Diff-heat-kernel}
{\rm (See \cite[Proposition~2.16]{DZ2}.)}~~~
Suppose $V\in {\rm RH}_q$ for some $q> n/2.$ There exists a nonnegative Schwartz  function $\varphi$ on $\Real^n$ such that
\begin{equation*}
\big| h_t(x-y)-{\mathcal K}_t(x,y) \big|\leq\left(\frac{\sqrt{t}}{\rho(x)}\right)^{2-n/q}\varphi_t(x-y),\quad x,y\in\Real^n,~t>0,
\end{equation*}
where $\varphi_t(x)=t^{-n/2}\varphi\left(x/\sqrt{t}\right)$.
\end{lemma}

 The Poisson semigroup associated to $\L$ can be obtained from the heat semigroup  through Bochner's subordination formula:
\begin{equation}\label{e2.6}
    e^{-t\sqrt{\L}}f(x)=\frac{1}{\sqrt{\pi}}\int_0^\infty\frac{e^{-u}}{\sqrt{u}}~e^{-{t^2\over 4u}\L}f(x)~du.
\end{equation}

From \eqref{e2.6},  the semigroup kernels ${\mathcal P}_t(x,y)$, associated to $e^{-t\sqrt{\L}}$, satisfy the following estimates.

\begin{lemma} \label{lem:Poisson-kernels}{\rm (\cite[Proposition 3.6]{MSTZ})}
Suppose $V\in {\rm RH}_q$ for some $q> n/2.$ For any $0<\delta<\min\left\{1, 2-\frac{n}{q}\right\}$ and every $N>0$, there exists a constant $C=C_{N}$ such that
 \begin{itemize}

\item[(i)]
 $\displaystyle
\left| {\mathcal P}_t(x,y)\right|\leq C{t \over (t^2+|x-y|^2)^{n+1\over 2}} \left(1+ {(t^2+|x-y|^2)^{1/2}\over \rho(x)} +{(t^2+|x-y|^2)^{1/2}\over \rho(y)}\right)^{-N}$;

\medskip

\item[(ii)]
For every $m\in \mathbb{N}=\{1,2,3,\cdots\}$,
$$
\left|t^m\partial^m_t{\mathcal P}_t(x,y) \right|\leq C {t^m \over (t^2+|x-y|^2)^{n+m\over 2}} \left(1+ {(t^2+|x-y|^2)^{1/2}\over \rho(x)}
+{(t^2+|x-y|^2)^{1/2}\over \rho(y)}\right)^{-N};
$$

\medskip

\item[(iii)]  For every $m\in \mathbb{N}$,
$$ \left|  t^m \partial_t^m e^{-t\sqrt{\L}}(1)(x)\right|\leq C\left(\frac{t}{ \rho(x)}\right)^{\delta}\left(1+{t\over \rho(x)}\right)^{-N}.
$$
 \end{itemize}
\end{lemma}

\bigskip

Moreover, combining Lemma~\ref{lem:Diff-heat-kernel} and \eqref{e2.6},  it's easy to verify, for each $x\in \mathbb{R}^n$,
\begin{equation}\label{eqn:aux-compare-Poisson}
	\left|e^{-t\sqrt{\L}}(1)(x) -1 \right|=\left|e^{-t\sqrt{\L}}(1)(x) -e^{-t\sqrt{-\Delta}} (1)(x) \right|\leq C \left(\frac{t}{ \rho(x)}\right)^{2-n/q}.
\end{equation}

\bigskip

For $s>0$, we define
$$
    {\Bbb F}(s):=\Big\{\psi:{\Bbb C}\to{\Bbb C}\ {\rm measurable}: \ \  |\psi(z)|\leq C {|z|^s\over ({1+|z|^{2s}})}\Big\}.
$$
Then for any non-zero function $\psi\in {\Bbb F}(s)$, we have that $\left(\int_0^{\infty}|{\psi}(t)|^2\frac{dt}{t}\right)^{1/2}<\infty$. Denote  $\psi_t(z):=\psi(tz)$ for $t>0$. It follows from the spectral theory in \cite{Yo} that for any $f\in L^2(\mathbb{R}^n)$,
\begin{align}\label{eqn:HFC-spectral}
    \Big\{\int_0^{\infty}\|\psi(t\sqrt{\L})f\|_{L^2(\mathbb{R}^n)}^2\,{dt\over t}\Big\}^{1/2} &=\Big\{\int_0^{\infty}\big\langle\,\overline{\psi}(t\sqrt{\L})\, \psi(t\sqrt{\L})f, f\big\rangle\, {dt\over t}\Big\}^{1/2}\nonumber\\
    &=\Big\{\big\langle \int_0^{\infty}|\psi|^2(t\sqrt{\L}) \,{dt\over t}f, f\big\rangle\Big\}^{1/2}\nonumber\\
    &\leq \kappa \|f\|_{L^2(\mathbb{R}^n)},
\end{align}
where $\kappa=\big\{\int_0^{\infty}|{\psi}(t)|^2\, {dt/t}\big\}^{1/2}$. The estimate  will be often used in this article.

\bigskip

\subsection{ ${\rm BMO}_{\L}$ spaces}\label{subsec:BMO-tent}
 The following characterization theorem for ${\rm BMO}_{\L}(\mathbb{R}^n )$ was  proved in \cite{DGMTZ}.
\begin{theorem}\label{thm:Carleson-BMO}
Let  $V\nequiv 0$ be  a nonnegative potential in $ {\rm RH}_q$, for some $q>n/2$. The following statements  are equivalent.
\begin{itemize}
\item [(i) ] $f$ is a function in ${\rm BMO}_{\L}(\mathbb{R}^n)$;

\smallskip

\item[(ii) ] $f\in L^2(\mathbb{R}^n, (1+|x|)^{-(n+\beta)}dx)$ for some $\beta>0$, and $\big\|t\sqrt{\L}e^{-t\sqrt{\L}}f(x) \big\|_{T_2^\infty}<\infty$;

\smallskip

\item [(iii)] $f\in L^p_{\rm Loc}(\mathbb{R}^n)$ and $\|f\|_{{\rm BMO}^p_{\L}}<\infty$,  where $1<p<\infty$ and
$$
     \|f\|_{{\rm BMO}^p_{\L}(\mathbb{R}^n)}:=\sup_{B}\left(\frac{1}{|B|} \int_{B}\left|f(y)-f_{B}\right|^p   d y\right)^{\frac{1}{p}}+\sup_{B:\, r_B\geq \rho(x_B)} \left(\frac{1}{|B(x_B,r_B)|} \int_{B(x_B,r_B)}|f(y)|^p\, d y\right)^{\frac{1}{p}};
$$

\smallskip

\item [(iv)] $f$ is in the dual space of $H_{\L}^1(\mathbb{R}^n)$.  Here, $H_{\L}^1(\mathbb{R}^n)$  is   defined by
$$
H_{\mathcal{L}}^{1}\left(\mathbb{R}^{n}\right)=\left\{f \in L^{1}\left(\mathbb{R}^{n}\right): \, \|f\|_{H_{\mathcal{L}}^{1}}:=\Big\|\sup _{t>0} \big|e^{-t \sqrt{\mathcal{L}}} f(x)\big|\,\Big\|_{ L^{1}}<\infty\right\}.$$
\end{itemize}

Moreover, the norms in above cases  are equivalent:
$$\|f\|_{{\rm BMO}_{\L}}\approx \big\|t\sqrt{\L}e^{-t\sqrt{\L}}f\big \|_{T_2^\infty}\approx \|f\|_{{\rm BMO}^p_{\L}}\approx \|f\|_{(H_{\mathcal{L}}^{1})^*}.$$
\end{theorem}

The following fact is used often below, which can be found in \cite[Lemma 2]{DGMTZ}.
\begin{lemma}\label{average on B}
There exists $C>0$ such that, for any  function $f\in {\rm BMO}_{\L}$ and any ball $B(x,r)$ of ${\mathbb R}^n$ with $r<\rho(x)$, then
$$
\left |\frac{1}{|B(x,r)|} \int_{B(x,r)} f(y) \, dy \right|\leq C \left(1+\log \frac{\rho(x)}{r}\right)\|f\|_{{\rm BMO}_{\L}}.
$$
\end{lemma}

\begin{remark}
 For further theory of  ${\rm BMO}$ and $\rm CMO$ spaces associated to differential operators, we refer the reader to \cite{A,ADM,AMR,CL,DuLi, DY1,DY2,HM,HLMMY,JY1,Ky,HWL} and the references  therein.
\end{remark}

\bigskip




\section{Proof of Theorem ~B}\label{sec:BMO-Carleson}
\setcounter{equation}{0}

\smallskip

In this section, we will show an equivalent characterization for ${\rm CMO}_{\L}(\mathbb{R}^n)$ by using tent spaces.

\begin{proof}[Proof of  Theorem B]
  Suppose $f\in {\rm CMO}_{\L}(\mathbb{R}^n)$, then   $f\in {\rm BMO}_{\L}(\mathbb{R}^n)$.  By Theorem~\ref{thm:Carleson-BMO}, we have that
$ t\sqrt{\L} e^{-t\sqrt{\L}}f(x) \in T_2^\infty$ and  $\|  t\sqrt{\L}e^{-t\sqrt{\L}}f \|_{T_2^\infty}\approx \|f\|_{{\rm BMO}_{\L}}$.  We will  prove $\eta_1 \big(t\sqrt{\L} e^{-t\sqrt{\L}}f\big)=\eta_2\big(t\sqrt{L} e^{-t\sqrt{\L}}f\big)=\eta_3\big(t\sqrt{\L} e^{-t\sqrt{\L}}f\big)=0$, where $\left\{\eta_i\big(t\sqrt{\L} e^{-t\sqrt{\L}}f\big)\right\}_{i=1}^3$ are defined in Section~\ref{sec:2}.

\medskip

To this end, we will prove that there exists a positive constant $c>0$ such that, for any ball $B=B(x_B,r_B)$,
\begin{equation}\label{eqn:key-aux}
    \left(\frac{1}{|B|}\iint_{\widehat{B}} \left|t\sqrt{\L} e^{-t\sqrt{\L}}f(x)\right|^2 \frac{dx\,dt}{t}\right)^{1/2}\leq c \sum_{k=0}^\infty 2^{-k} \delta_k(f,B),
\end{equation}
where
\begin{equation}\label{def:delta-k}
   \delta_k(f,B)=\sup_{B':B'\subseteq 2^{k+2}B,\, r_{B'}\in [2^{-1}r_B, 2r_B]} \left(\frac{1}{|B'|}\int_{B'} \left|\left(I-e^{-r_{B'}\sqrt{\L}}\right) f(x)\right|^2 dx\right)^{1/2}.
\end{equation}

Once the estimate \eqref{eqn:key-aux} is proved, $F:=t\sqrt{\L} e^{-t\sqrt{\L}}f\in T_{2,C}^\infty$ follows readily. Concretely, it is clear that for any $k=0,1,\cdots$, there holds $\delta_k (f,B)\leq \|f\|_{\widetilde{{\rm BMO}_{\L}}(\mathbb{R}^n)} \approx \|f\|_{{\rm BMO}_{\L}(\mathbb{R}^n)}$, where the $\widetilde{{\rm BMO}_{\L}}-$norm is given in \eqref{eqn:def-BMOL-norm-II}. Moreover, one may apply the condition $\gamma_1(f)=\gamma_2(f)=\gamma_3(f)=0$ to obtain that for any given $k$,
\begin{equation}\label{eqn:delta-k}
	\lim_{a\to 0}\sup_{B:\, r_B\leq a} \delta_k(f,B)=\lim_{a\to +\infty}\sup_{B:\, r_B\geq a} \delta_k(f,B)=\lim_{a\to +\infty} \sup_{B:\, B\subseteq \left(B(0,a)\right)^c} \delta_k(f,B)=0.
\end{equation}
It follows from \eqref{eqn:key-aux}  that
\begin{align*}
	\left(\frac{1}{|B|} \iint_{\widehat{B}} \left|t\sqrt{\L} e^{-t\sqrt{\L}} f(x)\right|^2  \frac{dx\,dt}{t} \right)^{1/2} &\leq c \sum_{k=0}^{\kappa_0} 2^{-k }\delta_k (f,B)+c \sum_{k=\kappa_0+1}^\infty 2^{-k} \|f\|_{{\rm BMO}_{\L}} \\
	& \leq  c \sum_{k=0}^{\kappa_0} 2^{-k }\delta_k (f,B)+c  2^{-\kappa_0} \|f\|_{{\rm BMO}_{\L}}.
\end{align*}
Note that if $\kappa_0$ is large enough, then the quantity $2^{-\kappa_0}\|f\|_{{\rm BMO}_{\L}}$ is sufficiently small. Fix a $\kappa_0$, we then use the property \eqref{eqn:delta-k} to obtain $\eta_1(F)=\eta_2(F)=\eta_3(F)=0$, as desired.

It suffices to prove estimate \eqref{eqn:key-aux}.  As  observed in \cite{HM},  we rewrite
\begin{align*}
	   f&=\frac{1}{r_B}\int_{r_B}^{2r_B} \left(I-e^{-s\sqrt{\L}}\right) f ds+\frac{1}{r_B}\int_{r_B}^{2r_B} e^{-s\sqrt{\L}}f ds\\
	   &= \frac{1}{r_B}\int_{r_B}^{2r_B} \left(I-e^{-s\sqrt{\L}}\right) f ds +(r_B\sqrt{\L})^{-1}e^{-r_B\sqrt{\L}}\left(I-e^{-r_B\sqrt{\L}}\right) f
\end{align*}
for any $B\subset {\mathbb R}^n$.  Then
\begin{align*}
	   \text{\rm LHS of }\eqref{eqn:key-aux} \leq\, & \sup_{s\in [r_B,2r_B]} \left(\frac{1}{|B|} \iint_{\widehat{B}}\left|  t\sqrt{\L} e^{-t\sqrt{\L}} \left(I-e^{-s\sqrt{\L}}\right)f(x)\right|^2 \frac{dx\,dt}{t} \right)^{1/2}  \\
	   &\qquad\qquad\qquad+ \left(\int_0^{r_B} \frac{t^2}{{r_B}^2} \big\|e^{-(t+r_B)\sqrt{\L}}\left(I-e^{-r_B\sqrt{\L}}\right)f\big\|_{L^\infty (B)}^2 \frac{dt}{t}\right)^{1/2}.
\end{align*}
 For any given $s\in [r_B,2r_B]$,  let
 $$F_{s,0}(y):=\chi_{2B}(y)\left(I-e^{-s\sqrt{\L}}\right)  f(y) \  {\rm and } \  F_{s,k}(y):=\chi_{2^{k+1}B\setminus {2^k B}}(y)\left(I-e^{-s\sqrt{\L}}\right)  f(y) \  {\rm for } \ k\geq 1  \ {\rm and} \  y\in \mathbb{R}^n.$$
Then
\begin{align}\label{eqn:key-aux-est}
	   \text{\rm LHS of }\eqref{eqn:key-aux}\leq\, & \sup_{s\in [r_B,2r_B]} \left(\frac{1}{|B|} \iint_{\mathbb{R}_+^{n+1}}\left|  t\sqrt{\L} e^{-t\sqrt{\L}} F_{s,0}(x)\right|^2 \frac{dx\,dt}{t} \right)^{1/2} \nonumber \\
& + \sup_{s\in [r_B,2r_B]}  \sum_{k= 1}^\infty \left(\int_0^{r_B}\left\|  t\sqrt{\L} e^{-t\sqrt{\L}} F_{s,k}\right\|_{L^\infty(B)}^2 \frac{dt}{t} \right)^{1/2} \nonumber \\
	   & + \sum_{k=0}^\infty\left(\int_0^{r_B} \frac{t^2}{{r_B}^2} \big\|e^{-(t+r_B)\sqrt{\L}}F_{r_B, k}\big\|_{L^\infty (B)}^2 \frac{dt}{t}\right)^{1/2}\nonumber\\
=:& I_0+\sum_{k=1}^\infty I_k+\sum_{k=0}^\infty II_k.
\end{align}

By \eqref{eqn:HFC-spectral}, we have
$$
     I_0\leq \sup_{s\in [r_B,2r_B]}\left(    \frac{1}{|B|}\int_{2B} \left|\left(I-e^{-s\sqrt{\L}}\right)  f(x)\right|^2 dx\right)^{1/2}.
$$
Note that there exists a positive constant $N_0=N_0(n)$ such that for every fixed $s\in [r_B,2r_B]$, the ball $2B$ can be covered by  finite-overlapped balls $\{B(x_i, s)\}_{i=1}^{N_0}$, where each $B(x_i, s)\subseteq 4B$. Hence,
$$
   I_0\leq C \delta_0(f,B).
$$

For any $s\in [r_B, 2r_B]$, $x\in B$ and $k\geq 1$, it follows from (ii) of Lemma \ref{lem:Poisson-kernels} that
\begin{align*}
  \left|t\sqrt{\L}e^{-t\sqrt{\L}}F_{s,k}(x)\right| & \leq  C \frac{t}{2^{k(n+1)}r_B} \frac{1}{|B|}\int_{2^{k+1}B} \left|\left(I-e^{-s\sqrt{\L}}\right)f(y)\right|dy.
\end{align*}
Moreover, it can be verified that for any ball $B(x_B,2^{k+1}r_B)$, there exists a corresponding collection of balls $B^{(k)}_{1},\, B^{(k)}_{2},\ldots , B^{(k)}_{{N_k}}$ such that
\begin{itemize}
  \item [(a) ] each ball $B^{(k)}_{i}$ is of the radius $s$ and $B^{(k)}_{i}\subseteq B(x_B,(2+2^{k+1})r_B)\subseteq B(x_B,2^{k+2}r_B)$;
  \smallskip
  \item [(b) ] $B(x_B,2^{k+1}r_B)\subseteq \bigcup_{i=1}^{N_k} B^{(k)}_{i}$;
    \smallskip
  \item[(c) ] there exists a constant $c>0$ independent of $k$ such that $N_k\leq c 2^{kn}$;
\smallskip
  \item[(d) ] $\sum_{i=1}^{N_k} \chi_{B^{(k)}_{i}}(x)\leq K$ for each $x\in B(x_B, 2^{k+1}r_B)$, where $K$ is independent of $k$.
    \smallskip
\end{itemize}

From the properties (a) -- (d) above, we may apply H\"older's inequality to obtain
\begin{align*}
  \left|t\sqrt{\L}e^{-t\sqrt{\L}}F_{s,k}(x)\right| & \leq C\frac{t}{2^{k(n+1)}r_B} \sum_{i=1}^{N_k}\left( \frac{1}{|B_{k_i}|} \int_{B_{k_i}} \left| \left(I-e^{-s\sqrt{\L}}\right)f(y) \right|^2 dy\right)^{1/2}\\
  &\leq C \frac{t}{2^{k}r_B}  \delta_k(f,B),
\end{align*}
which gives
$$
    I_k\leq C  \left(  \int_0^{r_B} \frac{t^2}{{r_B}^2} \frac{dt}{t}\right)^{1/2}2^{-k} \delta_k(f,B)  \leq C2^{-k}\delta_k(f,B).
$$

\smallskip

A similar argument  can be used to show $ II_k\leq C2^{-k}\delta_k(f,B)$ for $k\geq 0$, by noting that  the kernel $\mathcal{P}_{t+r_B}(x,y)$ of $e^{-(t+r_B)\sqrt{\L}}$, where  $0\leq t\leq r_B$,  satisfies
$$
     \left|\mathcal{P}_{t+r_B}(x,y)\right|\leq
     \begin{cases}
       C{r_B}^{-n}  & \mbox{if } |x-y|<2r_B, \\
         C r_B |x-y|^{-(n+1)}  & \mbox{otherwise},
     \end{cases}
$$
which  follows from (i) of Lemma \ref{lem:Poisson-kernels}.

 Plugging all estimates of terms $I_k$ and $II_k$ ($k\geq 0$) into \eqref{eqn:key-aux-est}, we deduce   \eqref{eqn:key-aux}, and then  get $ t\sqrt{\L} e^{-t\sqrt{\L}}f(x) \in T_{2,C}^\infty$.

\medskip

Conversely,  suppose that  $f\in L^2(\mathbb{R}^n, (1+|x|)^{-(n+\beta)}dx)$ for some $\beta>0$ and $ t\sqrt{\L} e^{-t\sqrt{\L}}f(x) \in T_{2,C}^\infty$. Let us prove $f\in {\rm CMO}_{\L}$.  In fact,
by applying  an  argument similar to  that of  \cite[Proposition~5.1]{DY2},  one can prove that the following identity
\begin{equation}\label{eqn:identity-reproducing}
    \int_{\mathbb{R}^n} f(x)g(x) \,dx=4 \iint_{\mathbb{R}_+^{n+1}} \left(t\sqrt{\L}e^{-t\sqrt{\L}}f\right)(x) \left(t\sqrt{\L}e^{-t\sqrt{\L}}g\right)(x)\frac{dx\,dt}{t}
\end{equation}
holds for any $f\in {\rm BMO}_{\L}$ and $g\in H_{\L}^1\cap L^2$.  Then the aimed result of $f\in {\rm CMO}_{\L}$  easily follows by a simple modification of    \cite[Proposition 3.3]{DDSTY} in which
 the  representation formula  (3.17) is replaced by \eqref{eqn:identity-reproducing}.  We have completed the proof of  Theorem B.
\end{proof}

\bigskip




\section{Proof of Theorem ~C}\label{sec:CMO-Schrodinger}
\setcounter{equation}{0}

\smallskip

Due to (ii) of Theorem~\ref{thm:CMO-dual-known}, the main difficulty of showing Theorem~C is to prove the implication (c) $\Rightarrow$ (a) of Theorem~C. That is, for any given $f\in \mathcal{B}_{\L}$,  one needs to approximate it in ${\rm BMO}_{\L}$ norm by using $C_c^\infty(\mathbb{R}^n)$ functions. To this end, we first make an attempt to show Theorem~C by using  a standard mollifier. Such an approach  has been successfully applied to character a local version of ${\rm CMO}(\mathbb{R}^n)$, which can be regarded as  ${\rm CMO}_{-\Delta+1}$;  see \cite{D} for details. However,  Lemma~\ref{lem:CMO-identity-approx} below tells us  this approach  is not completely effective to approximate a given function in $\mathcal{B}_{\L}$ directly, which in turn  reveals a certain difference between ${\rm CMO}_{-\Delta+1}$ and ${\rm CMO}_{-\Delta+V}$, due to the lack of uniform bounds for the variable function $\rho(x)$ defined by $V$.

\medskip

Let $\phi\in C_c^\infty(\mathbb{R}^n)$ be a radial bump function satisfying:
\begin{equation}\label{eqn:bump}
   {\rm supp}\, \phi\subseteq B(0,1),\quad        0\leq \phi\leq 1  \quad {\rm and }\quad \int \phi(x)\, dx=1.
\end{equation}
Let  $\phi_t(x):=t^{-n}\phi(x/t)$ for every $x\in \mathbb{R}^n$ and $t>0$. For any given $f\in L_{\rm loc}^1(\mathbb{R}^n)$, define
\begin{equation}\label{eqn:approx-operator}
     A_t(f)(x):=\phi_t * f(x)=\int_{\mathbb{R}^n} \phi_t(x-y)f(y)\, dy,\    \   x\in \mathbb{R}^n.
\end{equation}
 It's clear that $A_t(f)\in C^\infty(\mathbb{R}^n)$.  For $z\in \mathbb{R}^n$, denote $\tau_z(B):=\{x-z:\, x\in B\}$. We have the following lemma.

\begin{lemma}\label{lem:CMO-identity-approx}
Suppose $V\in {\rm RH}_q$ for some $q> n/2$.
 Let $f\in \mathcal{B}_{\L}$, where $\mathcal{B}_{\L}$ is the space defined in Theorem~C. Let $\phi$ and $A_t(f)$  be given in  \eqref{eqn:bump} and \eqref{eqn:approx-operator}, respectively.
 \begin{itemize}
   \item [(i) ] $A_t(f)$ is uniformly continuous on $\mathbb{R}^n$ and
   $\widetilde{\gamma}_1(A_t (f))=\widetilde{\gamma}_2(A_t (f))=\widetilde{\gamma}_3(A_t (f))=\widetilde{\gamma}_4(A_t (f))=0$  for each $0<t<1$.

   \smallskip

   \item [(ii) ] For any $\varepsilon>0$, there exist positive constants $R>>1$ and $t_0<<1$ such that for all $0<t<t_0$,
 \begin{subequations}
\begin{equation}\label{eqn:BMO-approx-a}
      \|A_t(f)-f\|_{{\rm BMO}}<\varepsilon,
\end{equation}
and
\begin{equation}\label{eqn:BMO-approx-b}
      \|A_t(f)-f\|_{{\rm BMO}_{\L}(\mathbb{R}^n)}\leq \varepsilon+ \sup_{\substack{B(x_B,r_B)\subseteq (B(0,(2c+2)R))^c\\  \rho(x_B) <r_B<R, \,\rho(x_B)<t }}\sup_{|z|\leq t}
   \left(\frac{1}{|\tau_z(B)|}\int_{\tau_z(B)} |f(x)|^2dx\right)^{1/2},
\end{equation}
\noindent where $c$ is  the constant in Lemma~\ref{lem:size-rho}.
\end{subequations}

\smallskip

\item [(iii)] If $f\in \mathcal{B}_{\L}$ with compact support, then $A_t(f)\in C_c^\infty(\mathbb{R}^n)$ and so $A_t(f)\in \mathcal{B}_{\L}$. Also,
$$
     \lim\limits_{t\to 0} \|A_t(f)-f\|_{{\rm BMO}_{\L}}=0.
$$
 \end{itemize}
\end{lemma}

\begin{proof}

(i). The uniform continuity of $A_t(f)$ for  $f\in {\rm BMO}$ was first proved by Dafni in \cite{D}. Of course, it holds for $f\in \mathcal{B}_{\L}$ since $\mathcal{B}_{\L}\subseteq {\rm BMO}$.

 For any $t>0$ and ball $B:=B(c_0,r_0)\subseteq \mathbb{R}^n$,  One may apply Minkowski's inequality and the fact $\int \phi_t(x) dx =1$ to obtain
\begin{align}\label{eqn:aux-approx-1}
    \left(\frac{1}{|B|}\int_B \left|A_t(f)(x)-\left(A_t(f)\right)_B\right|^2 dx\right)^{1/2} & = \left(\frac{1}{|B|}\int_B \left| \int_{\mathbb{R}^n} \left(f(x-z)-\frac{1}{|B|}\int_B f(y-z) \,dy\right) \phi_t(z)   \right|^2 dx\right)^{1/2}   \nonumber\\
    &\leq \int_{\mathbb{R}^n} \left(\frac{1}{|B|} \int_B \left| f(x-z)-\frac{1}{|B|}\int_B f(y-z)dy\right|^2 d x \right)^{1/2}  \phi_t(z) \, dz   \nonumber\\
    & =\int_{\mathbb{R}^n} \left(\frac{1}{|\tau_z (B)|} \int_{\tau_z(B) }  |f(x)-f_{\tau_z(B)} |^2 dx \right)^{1/2} \phi_t(z) \, dz\nonumber\\
    &\leq \sup_{|z|\leq t} \left(\frac{1}{|\tau_z (B)|}\int_{\tau_z(B)} |f(x)-f_{\tau_z(B)}|^2 dx \right)^{1/2}.
\end{align}
Similarly,
\begin{equation}\label{eqn:aux-approx-2}
    \left|\left(A_t(f)\right)_B\right|\leq \left[\left(|A_t(f)|^2\right)_B\right]^{1/2} \leq \sup_{|z|\leq t} \left(\frac{1}{\tau_z (B)}\int_{\tau_z(B)}|f(x)|^2 dx\right)^{1/2}.
\end{equation}
Combining these two estimates  and  $f\in \mathcal{B}_{\L}$, we can verify  directly $\widetilde{\gamma}_1(A_t (f))=\widetilde{\gamma}_2(A_t (f))=\widetilde{\gamma}_3(A_t (f))=0$.
If $r_0\geq \max\{a, \rho(c_0)\}$ and $a>>1$, then  $\tau_z (B)\subseteq 2B$ for any $|z|<1$. Thus, for any $0<t<1$, one has
\begin{align*}
\widetilde{\gamma}_4(A_t (f))&\leq\lim\limits_{a\to \infty}\sup_{r_0\geq \max\{a, \rho(c_0)\}}\sup_{|z|\leq t} \left(\frac{1}{|\tau_z (B)|}\int_{\tau_z(B)} |f(x)|^2 dx \right)^{1/2}\\
&\leq 2^{n/2}\lim\limits_{a\to \infty}\sup_{r_0\geq \max\{a, \rho(c_0)\}} \left(\frac{1}{|2B|}\int_{2B} |f(x)|^2 dx \right)^{1/2}=0.
\end{align*}
However, we may not obtain $\widetilde{\gamma}_5(A_t (f))=0$.

\bigskip

(ii).  Now we start to prove estimates \eqref{eqn:BMO-approx-a} and \eqref{eqn:BMO-approx-b}.
Note that for any  $\varepsilon>0$, it follows from  $\widetilde{\gamma}_i(f)=0$ for $1\leq i\leq 5$ that there exist positive constants $\delta<<1$ and $R>>1$ such that
 \begin{subequations}
\begin{equation}\label{eqn:BMO1-approx-a}
   \sup_{B:\, r_B\leq \delta} \left(\frac{1}{|B|}\int_B |f(x)-f_B|^2 dx\right)^{1/2}<{\varepsilon},
\end{equation}
\begin{equation}\label{eqn:BMO1-approx-b}
    \sup_{B:\, r_B\geq  R} \left(\frac{1}{|B|}\int_B |f(x)-f_B|^2 dx\right)^{1/2}<{\varepsilon},
\end{equation}
\begin{equation}\label{eqn:BMO1-approx-c}
    \sup_{B:\, B\subseteq (B(0,R))^c} \left(\frac{1}{|B|}\int_B |f(x)-f_B|^2 dx\right)^{1/2}<{\varepsilon},
\end{equation}
\begin{equation}\label{eqn:BMO1-approx-d}
    \sup_{B=B(x_B,r_B):\, r_B\geq  \max\{R,\,\rho(x_B)\}} \left(\frac{1}{|B|}\int_B |f(x)|^2 dx\right)^{1/2}<\varepsilon,
\end{equation}
and
\begin{equation}\label{eqn:BMO1-approx-e}
    \sup_{B=B(x_B,r_B):\, B\subseteq (B(0,R))^c,\, r_B\geq \rho(x_B)} \left(\frac{1}{|B|}\int_B |f(x) |^2 dx\right)^{1/2}<{\varepsilon}.
\end{equation}
 \end{subequations}

 \smallskip

 We start by proving \eqref{eqn:BMO-approx-a}.  For any fixed ball $B_0:=B(x_0,r_0)$, let us consider the following cases.

\smallskip

 \noindent{\it Case 1}. $r_0\leq \delta$. In this case, one may apply \eqref{eqn:aux-approx-1} and \eqref{eqn:BMO1-approx-a} that for any $0<t<1$,
\begin{align*}
    \left(\frac{1}{|B_0|}\int_{B_0} \left|A_t(f)(x)-f(x)-\left(A_t(f)-f\right)_{B_0}\right|^2 dx\right)^{1/2}
  \leq 2\sup_{|z|\leq t}  \left(\frac{1}{|\tau_z(B_0)|}\int_{\tau_z(B_0)} \left|f(x)-f_{\tau_z(B_0)}\right|^2 dx\right)^{1/2}<2\varepsilon.
\end{align*}

\smallskip

\noindent{\it Case 2}. $r_0>\delta$ and $B_0\cap B(0,2R)\neq \emptyset$. In this case, we just need to consider two subcases as follows.

\smallskip

{\it Subcase 2-1}. $\delta<r_0<R$. In this subcase, $B_0\subseteq  B(0,4R)$ and so
\begin{align*}
   \left(\frac{1}{|B_0|}\int_{B_0} \left|A_t(f)(x)-f(x)-\left(A_t(f)-f\right)_{B_0}\right|^2 dx\right)^{1/2}& \leq 2   \left(\frac{1}{|B_0|}\int_{B_0} \left|A_t(f)(x)-f(x)\right|^2 dx\right)^{1/2}\\
 &\leq \frac{2}{\delta^n} \|A_t(f)-f\|_{L^2(B(0,4R))}.
\end{align*}

Notice that $f\in L_{\rm loc}^2(\mathbb{R}^n)$ and $\big\{\phi_t\big\}_{0<t<1}$ is an approximate identity as $t\to 0$, there exists a constant $t_\varepsilon>0$ small enough such that
\begin{equation}\label{eqn:condition-t-1}
     \|A_t(f)-f\|_{L^2(B(0,4R))}< \frac{\delta^n \varepsilon}{2}\quad {\rm for} \quad 0<t<t_\varepsilon.
\end{equation}
From the above,  we have
$$
     \left(\frac{1}{|B_0|}\int_{B_0} \left|A_t(f)(x)-f(x)-\left(A_t(f)-f\right)_{B_0}\right|^2 dx\right)^{1/2} <\varepsilon \quad {\rm for} \quad 0<t<t_\varepsilon.
$$

\smallskip

{\it Subcase 2-2}. $r_0>R$.  It follows from \eqref{eqn:aux-approx-1} and \eqref{eqn:BMO1-approx-b} that
\begin{align*}
  \left(\frac{1}{|B_0|}\int_{B_0} \left|A_t(f)(x)-f(x)-\left(A_t(f)-f\right)_{B_0}\right|^2 dx\right)^{1/2}
  & \leq 2\sup_{|z|\leq t}  \left(\frac{1}{|\tau_z(B_0)|}\int_{\tau_z(B_0)} \left|f(x)-f_{\tau_z(B_0)}\right|^2 dx\right)^{1/2}\\
  & \leq 2 \sup_{B:\, r_B\geq R} \left(\frac{1}{|B|}\int_B |f(x)-f_B|^2 dx\right)^{1/2}<2\varepsilon.
\end{align*}

\smallskip

\noindent{\it Case 3}.  $r_0>\delta$ and $B_0\cap B(0,2R)=\emptyset$. In this case, it's clear $\tau_z(B_0)\subseteq (B(0,R))^c$ for $0<t<1$ since $R>0$ is sufficiently large. This, combined with \eqref{eqn:aux-approx-1} and \eqref{eqn:BMO1-approx-c}, deduces that
 \begin{align*}
  \left(\frac{1}{|B_0|}\int_{B_0} \left|A_t(f)(x)-f(x)-\left(A_t(f)-f\right)_{B_0}\right|^2 dx\right)^{1/2}
  & \leq 2\sup_{|z|\leq t}  \left(\frac{1}{|\tau_z(B_0)|}\int_{\tau_z(B_0)} \left|f(x)-f_{\tau_z(B_0)}\right|^2 dx\right)^{1/2}\\
  & \hskip -2cm \leq 2 \sup_{B:\, B\subseteq (B(0,R))^c} \left(\frac{1}{|B|}\int_B |f(x)-f_B|^2 dx\right)^{1/2}<2\varepsilon.
\end{align*}

Combining estimates in Cases 1-3, we obtain \eqref{eqn:BMO-approx-a}, as desired.

\smallskip

It remains to  verify the estimate \eqref{eqn:BMO-approx-b}. Let $\varepsilon,\, R$ be constants in \eqref{eqn:BMO1-approx-a} -- \eqref{eqn:BMO1-approx-e}.
For any $B_0=B(x_0,r_0)$ satisfying $r_0\geq \rho(x_0)$, consider the following cases.

\smallskip

{\it Case I}. $r_0\geq R$, i.e., $r_0\geq \max\{R, \, \rho(x_0)\}$. In this case,  $\tau_z (B_0)\subseteq 2B_0$ for any $|z|<1$.
This, together with \eqref{eqn:aux-approx-2} and \eqref{eqn:BMO1-approx-d}, deduces that
\begin{align*}
   \left(\frac{1}{|B_0|}\int_{B_0} \left|A_t(f)(x)-f(x) \right|^2 dx\right)^{1/2}&\leq 2\sup_{|z|\leq 1} \left(\frac{1}{|\tau_z (B_0)|}\int_{\tau_z (B_0)}|f(x)|^2\right)^{1/2} \\
   &\leq  2^{n+1}\sup_{|z|\leq 1} \left(\frac{1}{|2B_0|}\int_{2 B_0}|f(x)|^2\right)^{1/2}\\
   &\leq 2^{n+1}\sup_{B(x_B,r_B):\, r_B\geq \max\{R, \, \rho(x_B)\}}  \left(|B|^{-1} \int_{B}\left|f(x) \right|^{2} d x\right)^{1 / 2}\\
   &\lesssim\varepsilon.
\end{align*}

\smallskip

{\it Case II}. $\rho(x_0)\leq  r_0< R$. We need to consider the position of $B_0$.

\smallskip

{\it Subcase II-1}.    $B(x_0,r_0)\cap B(0,2(c+1)R)\neq \emptyset$. Then $B_0\subseteq B(0,2(c+2)R)$ due to $r_0\leq R$. Besides, by Lemma~\ref{lem:size-rho},
$$
     r_0>\rho(x_0)\geq c^{-1}\left\{1+\frac{|x_0|}{\rho(0)}\right\}^{-k_0}\rho(0)\geq c^{-1}\left\{1+\frac{(2c+4)R}{\rho(0)}\right\}^{-k_0}\rho(0):=C_{R,\,\rho(0)}.
$$
Note that there exists a constant $\widetilde{t_\varepsilon}>0$ small enough such that
\begin{equation}\label{eqn:condition-t-2}
     \|A_t(f)-f\|_{L^2(B(0,(2c+4)R))}<  (C_{R,\,\rho(0)})^n \varepsilon \quad {\rm for} \quad 0<t<\widetilde{t_\varepsilon}.
\end{equation}
Hence,
$$
   \left(\frac{1}{|B_0|}\int_{B_0} \left|A_t(f)(x)-f(x)\right|^2 dx\right)^{1/2}
 \leq \frac{1}{{r_0}^n} \|A_t(f)-f\|_{L^2(B(0,(2c+4)R))}<\varepsilon \quad {\rm for} \quad 0<t<\widetilde{t_\varepsilon}.
$$

\smallskip

{\it Subcase II-2}.  $B(x_0,r_0)\cap B(0,2(c+1)R)= \emptyset$.
\begin{itemize}
  \item If $\rho(x_0)\geq t$, then it follows from Lemma~\ref{lem:size-rho} again to see that for any $|z|\leq t$,
    \begin{align*}
  \rho(x_0-z)\leq c\left (1+\frac{|z|}{\rho(x_0)}\right )^{k_0/(k_0+1)} \rho(x_0)\leq c\left(t+\rho(x_0)\right)\leq 2cr_0.
   \end{align*}
This, combined with the fact $2c\cdot \tau_z(B_0)\subseteq (B(0,R))^c$, allows us to apply \eqref{eqn:aux-approx-2} and
\eqref{eqn:BMO1-approx-e} to obtain
\begin{align*}
   \left(\frac{1}{|B_0|}\int_{B_0} \left|A_t(f)(x)-f(x) \right|^2 dx\right)^{1/2} &\leq 2\sup_{|z|\leq t} \left(\frac{1}{|\tau_z (B_0)|}\int_{2c\cdot \tau_z (B_0)}|f(x)|^2\right)^{1/2}\\
   &\leq 2 (2c)^n\sup_{B(x_B,r_B):\, B\subseteq (B(0,R))^c,r_B\geq \rho(x_B)}  \left(\frac{1}{|B|} \int_{B}\left|f(x) \right|^{2} d x\right)^{1 / 2}\lesssim\varepsilon.
\end{align*}
  \item Consider $\rho(x_0)<t$.  For any $|z|<t$, it holds
$
  \rho(x_0-z)\leq c{\rho(x_0)}^{\frac{1}{k_0+1}}\left(\rho(x_0)+t\right)^{\frac{k_0}{k_0+1}}.
$
Compared to  $t$, $\rho(x_0)$ may be much smaller. So it  fails to bound $\rho(x_0-z)$ for $|z|\leq t$ by  $C\cdot\rho(x_0)$. This is  why  the second term of RHS of \eqref{eqn:BMO-approx-b} appears.
\end{itemize}

From the above,  (ii) is  proved.

\medskip

(iii). It is a direct corollary of (i) and (ii).
\end{proof}

\begin{remark}
Consider $\mathcal{L}=-\Delta+1$. In this case  $\rho(x)$ is constant. By \eqref{eqn:aux-approx-1} and \eqref{eqn:aux-approx-2}, it can be seen that $f\in  {\mathcal B}_{\mathcal L}$ implies $A_t(f)\in {\mathcal B}_{\L}$.  Moreover,  $\lim\limits_{t\to 0}\|A_t(f)-f\|_{{\rm BMO}_{\mathcal L}}=0$ follows from \eqref{eqn:BMO-approx-b}. However, for $\mathcal{L}=-\Delta+V(x)$, in order to obtain the same result,  the additional condition that $f$ has compact support is needed.
\end{remark}

\smallskip

 Lemma~\ref{lem:CMO-identity-approx}  hints that before smoothing $f\in \mathcal{B}_{\L}$,  some data pre-processing  should be considered.  Recall that Uchiyama \cite[pp. 166-167]{U} gave an explicit construction to approximate  a  function in $\mathcal{B}$ satisfying \eqref{eqn:U-CMO-a} -- \eqref{eqn:U-CMO-c} by step functions.   We will use a modified Uchiyama's construction to approximate a  given function in $\mathcal{B}_\L$ by step functions with compact supports,  which   relies heavily on  the properties of the function $\rho$.  This result is useful for proving the aimed Theorem~C.

\smallskip

Let $Q:=Q(c_Q,\ell(Q))\subseteq \mathbb{R}^n$ be a cube of center $c_Q$ and
sidelength $\ell(Q)$. For any constant $c>0$,  denote $cQ:=Q(c_Q,c\ell(Q))$.
Observe the following facts:	for each ball $B\subseteq \mathbb{R}^n$, there exists a cube $Q\subseteq \mathbb{R}^n$ satisfying $B\supseteq  Q$ and $B\subseteq \sqrt{n}Q$. Furthermore,	 there exists a constant $C=C(n)>1$ independent of $B$ and $f$ such that
$$
	 \frac{1}{C}\left(\frac{1}{|Q|}\int_Q |f(x)-f_Q|^2  dx\right)^{1/2}  \leq \left(\frac{1}{|B|}\int_B |f(x)-f_B|^2  dx\right)^{1/2}   \leq  C \left(\frac{1}{|\sqrt{n}Q|}\int_{\sqrt{n}Q} |f(x)-f_{\sqrt{n}Q}|^2  dx\right)^{1/2}.
$$

Hence we can substitute cubes for balls (simultaneously replacing $r_B$ and $f_B$ by $\ell(Q)$ and $f_Q$, respectively) in the definitions of $\widetilde{\gamma}_i(f)$ for $f\in \mathcal{B}_{\L}$, where $1\leq i\leq 5$.

Therefore, for any given $f\in \mathcal{B}_{\L}$ and $\varepsilon>0$,  it follows from $\widetilde{\gamma}_1(f)=0$ and $\widetilde{\gamma}_i(f)=0$  ($2\leq i\leq 5$) respectively to see
there exist two integers $I_\varepsilon >>1$ and $J_\varepsilon>>1$ such that

 \begin{subequations}
\begin{equation}\label{eqn:CMO-a}
   \sup_{Q:\, \ell(Q)\leq 2^{-I_\varepsilon}} \left(\frac{1}{|Q|}\int_Q |f(x)-f_Q|^2 dx \right)^{1/2}<\frac{\varepsilon}{5\cdot 4^{n}},
\end{equation}
\begin{equation}\label{eqn:CMO-b}
    \sup_{Q:\, \ell(Q)\geq  2^{J_\varepsilon}} \left(\frac{1}{|Q|}\int_Q |f(x)-f_Q|^2 dx\right)^{1/2}<\frac{\varepsilon}{5\cdot 4^n},
\end{equation}
\begin{equation}\label{eqn:CMO-c}
    \sup_{Q:\, Q\subseteq (Q(0, 2^{J_\varepsilon +1}))^c} \left(\frac{1}{|Q|}\int_Q |f(x)-f_Q|^2 dx\right)^{1/2}<\frac{\varepsilon}{5\cdot 4^n},
\end{equation}
\begin{equation}\label{eqn:CMO-d}
    \sup_{Q:\, \ell(Q)\geq  \max\{2^{J_\varepsilon},\,\rho(c_Q)\}} \left(\frac{1}{|Q|}\int_Q |f(x)|^2 dx\right)^{1/2}<\frac{\varepsilon}{2},
\end{equation}
and
\begin{equation}\label{eqn:CMO-e}
    \sup_{ Q\subseteq (Q(0, 2^{J_\varepsilon+1}))^c,\, \ell(Q)\geq \rho(c_Q)} \left(\frac{1}{|Q|}\int_Q |f(x) |^2 dx\right)^{1/2}<\frac{\varepsilon}{2}.
\end{equation}
 \end{subequations}

Note that
\begin{equation}\label{eqn:dyadic-cube-k}
 \mathcal{R}_k:=Q(0,2^{k+1})=\left\{x=(x_1,\ldots,x_n)\in \mathbb{R}^n:\, |x_i|\leq 2^k \   {\rm for }\  1\leq i\leq n\right\}
\end{equation}
is a union of $2^n$ dyadic cubes of sidelength $2^k$ for each $k\in \mathbb{Z}$.  Besides, the set $\mathcal{R}_{k+1}\setminus \mathcal{R}_k$ can be divided into mutually disjoint dyadic cubes with sidelength $2^{l}$ for any fixed $l\leq k$. These, combined with properties \eqref{eqn:CMO-a} -- \eqref{eqn:CMO-e}, motivate us to give the following construction which is partly adapted from \cite{U}: throughout this proof, whenever we mention $Q_x$ for $x\in \mathbb{R}^n$, it always denotes the unique dyadic cube that contains $x$ as follows.
\begin{itemize}
	\item for $x\in \mathcal{R}_{J_\varepsilon}$, let $Q_x$ be the dyadic cube of sidelength $2^{-I_\varepsilon-2}$ that contains $x$;
	\item for $x\in \mathcal{R}_{m+1}\setminus \mathcal{R}_{m}$ whenever the integer $m\geq J_\varepsilon$, let $Q_x$ be the dyadic cube of sidelength $2^{m-I_\varepsilon-J_\varepsilon-1}$ that contains $x$.
\end{itemize}
Now  define
\begin{equation}\label{def:average-type-II}
     \mathcal{A}_{\varepsilon}(f)(x):=f_{Q_x}=\frac{1}{|Q_x|}\int_{Q_x} f(y)\, dy ,\  \   x\in \mathbb{R}^n.
\end{equation}
which,  together with  \eqref{eqn:CMO-a} and \eqref{eqn:CMO-c}, implies that
\begin{equation}\label{eqn:aux-local-average-bound}
        \left( \frac{1}{|Q_x|}\int_{Q_x}\left|f(y)-\mathcal{A}_\varepsilon(f)(y)\right|^2 dy\right)^{1/2}\leq \frac{\varepsilon}{5\cdot 4^n}   \quad {\rm for  \    all}\   \   Q_x\subseteq \mathbb{R}^n.
\end{equation}

\medskip

\begin{lemma}\label{lem:approx-type-II}
Suppose $V\in {\rm RH}_q$ for some $q> n/2$.
 Let $f\in \mathcal{B}_{\L}$, where $\mathcal{B}_{\L}$ is the space defined in Theorem~C.  For any $\varepsilon>0$, let  $I_\varepsilon$ and $J_\varepsilon$ be given in \eqref{eqn:CMO-a} -- \eqref{eqn:CMO-e}, and let $\mathcal{A}_{\varepsilon}(f)$ be the function defined in \eqref{def:average-type-II}.  Then
\begin{equation}\label{eqn:approx-type-II-0}
    \|f-\mathcal{A}_{\varepsilon}(f)\|_{\rm BMO}\lesssim \varepsilon.
\end{equation}
In addition,  assume that $I_\varepsilon$ is sufficiently large such that  $2^{-I_\varepsilon-1}\leq \inf\limits_{x\in \mathcal{R}_{J_\varepsilon+2}}\rho(x)$.  We have
\begin{align}\label{eqn:approx-type-II}
     \|f-\mathcal{A}_{\varepsilon}(f)\|_{{\rm BMO}_{\L}}\lesssim \varepsilon.
\end{align}
\end{lemma}

\begin{remark}\label{rem:workable}
We note that the assumption  $2^{-I_\varepsilon-1}\leq \inf\limits_{x\in \mathcal{R}_{J_\varepsilon+2}}\rho(x)$ is workable and non-contradictory. Indeed, using Lemma~~\ref{lem:size-rho},   it holds
$$
  \rho(x)\geq c^{-1}\Big(1+\frac{|x|}{\rho(0)}\Big)^{-k_0}\rho(0)\geq c^{-1}\Big(1+\frac{\sqrt{n}2^{J_\varepsilon+2}}{\rho(0)}\Big)^{-k_0}\rho(0), \ \ {\rm for \ any} \ x\in \mathcal{R}_{J_\varepsilon+2}.
$$
 Taking $2^{-I_\varepsilon-1}\leq c^{-1}\Big(1+\frac{\sqrt{n}2^{J_\varepsilon+2}}{\rho(0)}\Big)^{-k_0}\rho(0)$, yields the desired assumption.
\end{remark}

\begin{proof} [Proof of Lemma \ref{lem:approx-type-II}]
Firstly, we claim that  $\mathcal{A}_{\varepsilon}(f)$ has the following two  properties: there exists a positive integer $M_\varepsilon\gtrsim  I_\varepsilon+J_\varepsilon$ such that
	\begin{equation}\tag{\textbf{P1}}\label{P1}
	 \sup_{x\in \mathbb{R}^n\setminus \mathcal{R}_{M_{\varepsilon}}} |\mathcal{A}_\varepsilon(f)(x)|<\varepsilon/2,
	 \end{equation}
	and
\begin{equation}\tag{\textbf{P2}}\label{P2}
     \sup\left\{ \big|\mathcal{A}_{\varepsilon}(f)(x)-\mathcal{A}_{\varepsilon}(f)(y)\big|:\    \overline{Q_x}\cap \overline{Q_y}\neq \emptyset \right\} <\varepsilon,
\end{equation}
where $\overline{Q}$ is the closure of $Q$ in $\mathbb{R}^n$.
	
\smallskip

Let us prove the  above claim.    It follows from Lemma~\ref{lem:size-rho}  that for any $x\in \mathcal{R}_{m+1}\setminus \mathcal{R}_{m}$ with $m\geq  J_\varepsilon$,
$$
    \rho(x)\leq c \left\{1+\frac{|x|}{\rho(0)}\right\}^{\frac{k_0}{k_0 +1}} \rho(0)\leq C \cdot 2^{\frac{k_0}{k_0+1}m},
$$
where $C$ is a constant dependent  on $\rho(0)$.  Meanwhile, $\ell(Q_{x})=2^{m-I_\varepsilon-J_\varepsilon-1}$,  so one has
$$
    \rho(y)\leq \ell(Q_{x}) \quad {\rm for \   all} \  \    y\in Q_x,
    \quad  {\rm if}\  \   m\geq (k_0+1)\left(\log_2 {C}+I_\varepsilon +J_\varepsilon+1\right)=:M_{\varepsilon}.
$$
In particular, denote the center of $Q_x$ by $c_{Q_x}$, then $\rho(c_{Q_x})\leq  \ell(Q_x)$. Hence,  \eqref{P1} is a straightforward consequence of \eqref{eqn:CMO-e}.

 Then we  turn to  \eqref{P2}.  Suppose $\overline{Q_x}\cap \overline{Q_y}\neq \emptyset$.  Let  $Q_{x,y}$ be the smallest cube  that contains $Q_x$ and $Q_y$, and we remind that $Q_{x,y}$ may not be a dyadic cube. Assume $\ell(Q_x)\leq \ell(Q_y)$, then  it follows from the definition of  dyadic cubes $\{Q_z\}_{z\in \mathbb{R}^n}$ that $\ell(Q_x)=\ell(Q_y)/2$ if $\ell(Q_x)\neq \ell(Q_y)$ and  $|Q_{x,y}|\leq 3^{n}|Q_x|$.   Note that
$$
 \left |\mathcal{A}_{\varepsilon}(f)(x)- \mathcal{A}_{\varepsilon}(f)(y)\right |\leq |f_{Q_x}-f_{Q_{x,y}} |+|f_{Q_y}-f_{Q_{x,y}} |\leq  \frac{2\cdot 3^{n}}{|Q_{x,y}|}\int_{Q_{x,y}} |f(z)-f_{Q_{x,y}}| \,dz.
$$
It suffices to show \eqref{P2} in the following two cases. In the case of  $x, y\in \mathcal{R}_{J_\varepsilon+2}$,  we have that $\ell(Q_x)\leq \ell(Q_y)\leq 2^{-I_\varepsilon}$ and $\ell(Q_{x,y})\leq 2^{-I_\varepsilon+1}$.  Then \eqref{P2} follows from \eqref{eqn:CMO-a}.  In the case of  $x,y\notin  \mathcal{R}_{J_\varepsilon+1}$, we have $Q_{x,y}\subseteq (R_{J_\varepsilon})^c=(Q(0,2^{J_\varepsilon +1}))^c$.   Then \eqref{P2}  follows from \eqref{eqn:CMO-c}.

\medskip

With \eqref{P1} and \eqref{P2} at our disposal,   we now  show \eqref{eqn:approx-type-II-0} and \eqref{eqn:approx-type-II}.

Denote by $Q:=Q(c_Q, \ell(Q))$  the cube in $\mathbb{ R}^n$. Let's prove \eqref{eqn:approx-type-II-0} by
  considering the following cases.

\smallskip

\noindent{\it Case I.}  $\ell(Q)<\frac{1}{8}\max\{\ell(Q_x):\, Q_x\cap Q\neq \emptyset\}$.

\smallskip

 By the construction of $\{Q_z\}_{z\in \mathbb{R}^n}$, it is not difficult to show  the fact:  if $Q_x\cap Q \neq \emptyset$, $Q_y\cap Q\neq \emptyset$ and $\ell(Q)<\frac{1}{8}\max\{\ell(Q_x):\, Q_x\cap Q\neq \emptyset\}$,  then
 \begin{align}\label{estimate-size}
 \ell(Q_x)/\ell(Q_y)\in \left\{\frac{1}{2},1,2\right\}  \quad   {\rm and}  \quad     \overline{Q_x}\cap \overline{Q_y}\neq \emptyset.
\end{align}

\smallskip

One can compute
\begin{align}\label{estimate-0}
     &\left(\frac{1}{|Q|} \int_Q \left|f(x)-\mathcal{A}_\varepsilon(f)(x)   -\left(f-\mathcal{A}_\varepsilon(f)\right)_Q\right|^2 d x\right)^{1/2}\nonumber\\
     &\leq \left(\frac{1}{|Q|} \int_Q \left|\mathcal{A}_\varepsilon(f)(x)-\left(\mathcal{A}_\varepsilon(f)\right)_Q\right|^2 d x\right)^{1/2}+\left(\frac{1}{|Q|} \int_Q \left|f(x)-f_Q\right|^2 d x\right)^{1/2}.
\end{align}
Note that $x,y\in Q$ implies $Q_x, Q_y\in \{Q_z: Q_z\cap Q\neq \emptyset\}$  and by \eqref{estimate-size}  we get   $\overline{Q_x}\cap \overline{Q_y}\neq \emptyset$.  Then one may  apply \eqref{P1} and \eqref{P2} to obtain
\begin{align}\label{estimate-1}
    \left(\frac{1}{|Q|} \int_Q \left|\mathcal{A}_\varepsilon(f)(x)   -\left(\mathcal{A}_\varepsilon(f)\right)_Q\right|^2 d x\right)^{1/2}&\leq
     \left(\frac{1}{|Q|^2} \int_Q\int_Q \left|\mathcal{A}_\varepsilon(f)(x)-\mathcal{A}_\varepsilon(f)(y)\right|^2 d yd x\right)^{1/2}
     \leq \varepsilon.
\end{align}

Consider $\left(\frac{1}{|Q|} \int_Q \left|f(x)-f_Q\right|^2 d x\right)^{1/2}$.
\begin{itemize}
	\item if $Q\cap \mathcal{R}_{J_\varepsilon}\neq \emptyset$, then $\ell(Q)< \frac{1}{8}\max\{\ell(Q_x):\, Q_x\cap Q\neq \emptyset\}\leq 2^{-I_\varepsilon-2}$ . This allows us to apply \eqref{eqn:CMO-a} to obtain
	
$$
     \left(\frac{1}{|Q|} \int_Q \left|f(x)-f_Q\right|^2 d x\right)^{1/2}<\frac{\varepsilon}{5\cdot 4^{n}};
$$

\smallskip

\item if $Q\cap \mathcal{R}_{J_\varepsilon}=\emptyset$, then it follows from \eqref{eqn:CMO-c} that
$$
     \left(\frac{1}{|Q|} \int_Q \left|f(x)-f_Q\right|^2 d x\right)^{1/2}<\frac{\varepsilon}{5\cdot 4^{n}}.
$$
\end{itemize}
This, combined with \eqref{estimate-1}, implies
$$
\left(\frac{1}{|Q|} \int_Q \left|f(x)-\mathcal{A}_\varepsilon(f)(x)   -\left(f(x)-\mathcal{A}_\varepsilon(f)\right)_Q\right|^2 d x\right)^{1/2}<2\varepsilon.
$$

\medskip

\noindent{\it Case II.}   $\ell(Q)\geq \frac{1}{8}\max\{\ell(Q_x):\, Q_x\cap Q\neq \emptyset\}$.

In this case $\bigcup_{Q_x\cap Q\neq \emptyset}Q_x \subseteq  20Q$.
By \eqref{eqn:aux-local-average-bound}, one can write
\begin{align}\label{estimate-Qbig}
\left(\frac{1}{|Q|} \int_Q \left|f-\mathcal{A}_\varepsilon(f)   -\left(f-\mathcal{A}_\varepsilon(f)\right)_Q\right|^2 \,dx\right)^{1/2}&\leq 2\left(\frac{1}{|Q|}\int_{Q} \big|f(y)-\mathcal{A}_\varepsilon(f)(y)\big|^2 dy\right)^{1/2}\nonumber\\
& \leq 2 \left(\sum_{Q_x:\, Q_x\cap Q\neq \emptyset}\frac{|Q_x|}{|Q|}  \frac{1}{|Q_x|} \int_{Q_x} \big|f(y)-f_{Q_x}\big|^2 dy\right)^{1/2}\\
& \leq 2 \left(\frac{\varepsilon}{5\cdot 4^n}\right) \left(\frac{|\bigcup_{Q_x\cap Q\neq \emptyset} Q_x|}{|Q|}\right)^{1/2}\leq \frac{ 20^{n/2}}{  4^{n}} {\varepsilon},\nonumber
\end{align}
as desired.

\smallskip

 Combining the two cases above, we obtain that $\|f-\mathcal{A_\varepsilon}(f)\|_{\rm BMO}\lesssim \varepsilon$.

 \bigskip

Lastly, we prove \eqref{eqn:approx-type-II}. It suffices to prove
\begin{align}\label{estimate-BMOL}
     \sup\limits_{Q:\ell(Q)\geq \rho(c_Q)}\left(\frac{1}{|Q(c_Q,\ell(Q))|}\int_{Q(c_Q,\ell(Q))} |f-\mathcal{A}_\varepsilon(f)|^2 dx\right)^{1/2}\lesssim \varepsilon.
\end{align}

 If  $\ell(Q)\geq \frac{1}{8} \max\{\ell(Q_x):\, Q_x\cap Q\neq \emptyset\}$,   we may  use \eqref{estimate-Qbig} to obtain
$$
     \left(\frac{1}{|Q(c_Q,\ell(Q))|}\int_{Q(c_Q,\ell(Q))} |f(x)-\mathcal{A}_\varepsilon(f)(x)|^2 dx\right)^{1/2}\lesssim \varepsilon.
$$

\smallskip

It remains to consider the case of  $\ell(Q)<\frac{1}{8} \max\{\ell(Q_x):\, Q_x\cap Q\neq \emptyset\}$ and  $\ell(Q)\geq \rho(c_Q)$.  It follows from \eqref{estimate-size} that $\ell(Q)\leq \ell(Q_x)$ whenever $Q_x\cap Q\neq \emptyset$.  We claim that there holds
\begin{align*}
Q\cap \mathcal{R}_{J_\varepsilon+1}=\emptyset.
\end{align*}
In fact,  if $Q\cap  \mathcal{R}_{J_\varepsilon+1}\neq \emptyset$, then $\ell(Q)\leq 2^{-I_\varepsilon -2}$ and  $Q\subseteq  \mathcal{R}_{J_\varepsilon+2}$.  The  assumption  that $I_\varepsilon$ is sufficiently large such that  $2^{-I_\varepsilon-1}\leq \min_{x\in \mathcal{R}_{J_\varepsilon+2}}\rho(x)$, gives  $\ell(Q)< \rho(c_Q)$, which contradicts our condition.

 If  $Q_x\cap Q\neq \emptyset$,  then
	 $|c_Q- c_{Q_x}|\leq \sqrt{n} \left(\ell(Q_x)+\ell(Q)\right)/2\leq \sqrt{n} \ell(Q_x).$
   By using Lemma~\ref{lem:size-rho} and $\rho(c_Q)\leq \ell(Q)\leq \ell(Q_x)$,  one can compute
\begin{align*}
     \rho(c_{Q_x})& \leq c\left(\frac{\rho(c_Q)+\sqrt{n}\ell(Q_x)}{\rho(c_Q)}\right)^{k_0/(k_0+1)}  \rho(c_Q)
    \leq c\left(\sqrt{n}+1\right) \ell(Q_x).
\end{align*}
Denote $C_1:=c\left(\sqrt{n}+1\right)$ and  $ Q_x^*:=Q(c_{Q_x}, C_1\ell(Q_x))$. Clearly,   $\rho( c_{Q_x^*})\leq \ell( Q_x^*)$. The fact $Q\cap \mathcal{R}_{J_\varepsilon+1}=\emptyset$ implies that $Q_x^*\subseteq (\mathcal{R}_{J_\varepsilon})^c$, if $J_\varepsilon$ is chosen large enough.    For any $x\in Q$,  we  have
\begin{align*}
	|\mathcal{A}_\varepsilon(f)(x)|=|f_{Q_x}| \leq (C_1)^{n/2} \left(\frac{1}{| Q_x^*|}\int_{ Q_x^*}|f(y)|^2 dy\right)^{1/2} \lesssim \varepsilon,
\end{align*}
where in the last inequality  we  used \eqref{eqn:CMO-e}.
Hence $\|\mathcal{A}_\varepsilon(f)\|_{L^\infty(Q)}\lesssim \varepsilon$.
This, combined with \eqref{eqn:CMO-e} and the fact  $Q\subseteq (\mathcal{R}_{J_\varepsilon})^c$ and $\ell(Q)\geq \rho(Q)$, gives that
$$
     \left(\frac{1}{|Q|}\int_{Q} |f(x)-\mathcal{A}_\varepsilon(f)(x)|^2 dx\right)^{1/2}\leq \left(\frac{1}{|Q|}\int_{Q} |f(x)|^2 dx\right)^{1/2} +\|\mathcal{A}_\varepsilon(f)\|_{L^\infty (Q)}\lesssim \varepsilon.
$$

The proof of  Lemma~\ref{lem:approx-type-II} is completed.
\end{proof}

\medskip

We are now in a position to show Theorem~C, by combining  Lemma~\ref{lem:CMO-identity-approx}  and Lemma~\ref{lem:approx-type-II}.

\bigskip

\begin{proof}[Proof of Theorem~C]
The proof follows from the sequence of implications
$$
(a) \Rightarrow  (b)\Rightarrow (c) \Rightarrow (a)  \qquad {\rm and} \qquad (c)\Leftrightarrow (d).
$$

 The implication (a) $\Rightarrow $ (b) follows directly  from   (ii) of Theorem~\ref{thm:CMO-dual-known} and the fact ${C_c^\infty}({\mathbb R}^n)\subset C_0({\mathbb R}^n)$.

\medskip

 {\it Proof of  ``(b) $\Rightarrow$ (c)''.}
We first show that  $f\in C_0(\mathbb{R}^n)\Rightarrow f\in \mathcal{B}_{\L}$.

Observe that $f\in C_0(\mathbb{R}^n)$ implies  $f\in {\rm BMO}_{\L}(\mathbb{R}^n)$  by the simple fact  $C_0\subseteq L^\infty \subseteq {\rm BMO}_{\L}$.

 Note that
  $$\lim_{a\to 0}\sup _{B: \,r_{B} \leq a}\left(|B|^{-1} \int_{B}\left|f(x)-f_{B}\right|^{2} d x\right)^{1 / 2} \leq \lim_{a\to 0}\sup_{x,y\in B:\, r_B\leq a} |f(x)-f(y)|,$$
which, together with the uniform continuity of $f\in C_0$,  gives $ \widetilde{\gamma}_1(f)=0$.

    \smallskip

 Since  $f\in C_0$, then for  any given $\varepsilon>0$,   there exists a constant $N_\varepsilon>0$   such that  $|f(x)|\leq \varepsilon$  whenever $|x|\geq N_\varepsilon$.
For every ball $B$ with $r_B>a$, where $a>0$ is sufficiently large,  one has
\begin{align}\label{eqn:f-square-average}
    \left(|B|^{-1} \int_{B}\left|f(x) \right|^{2} d x\right)^{1 / 2}&\leq   \frac{\|f\|_{L^2(B(0,N_\varepsilon)\cap B)}}{|B|^{1/2}}+\|f\|_{L^\infty(B(0,N_\varepsilon)^c)}\frac{|B\setminus B(0,N_\varepsilon)|^{1/2}}{|B|^{1/2}}\nonumber\\
    &\lesssim \frac{\|f\|_{L^\infty}{N_\varepsilon}^{n/2}}{a^{n/2}} +\varepsilon.
\end{align}
This says that for any given $\varepsilon>0$, there exists $a=a\left(f, \varepsilon\right)$ sufficiently large, such that
$$
      \sup_{B:\, r_B\geq a}  \left(\frac{1}{|B|} \int_{B}\left|f(x) \right|^{2} d x\right)^{1 / 2} <2\varepsilon,
$$
which yields
$$
     \widetilde{\gamma}_2(f)(x)=0  \qquad {\rm and} \qquad  \widetilde{\gamma}_4(f)=0.
$$

\smallskip

By using $\displaystyle \lim_{|x|\to\infty }f(x)=0$, we have
$$
      \lim_{a\to \infty}\sup_{B:\, B\subset B(0,a)^c}  \left(\frac{1}{|B|} \int_{B}\left|f(x) \right|^{2} d x\right)^{1 / 2} \leq \lim_{a\to \infty} \sup_{|x|\geq a} |f(x)|=0,
$$
which gives
$$
    \widetilde{\gamma}_3(f)=0   \quad  {\rm and } \quad  \widetilde{\gamma}_5(f)=0.
$$
Thus,  we have shown $C_0\subseteq \mathcal{B}_{\L}$, as desired.

\smallskip

To complete the proof of $\overline{C_0}^{{\rm BMO}_{\L}}\subseteq \mathcal{B}_{\L}$, it suffices to show that $\mathcal{B}_{\L}$ is closed in ${\rm BMO}_{\L}$.  Suppose that  $f\in {\rm BMO}_{\L}$ and  $f_k \in \mathcal{B}_{\L}, \ k\in \mathbb{N}$,  satisfying  $\displaystyle \lim_{k\to \infty} \|f_k -f\|_{{\rm BMO}_{\L}}=0$.  We will prove $f\in \mathcal{B}_{\L}$.

For any ball $B\subseteq \mathbb{R}^n$ and $k\in {\mathbb N}$, it follows from Theorem~\ref{thm:Carleson-BMO} that
$$
	 \left(\frac{1}{|B|}\int_B |f(x)-f_B|^2\,dx\right)^{1/2} \lesssim \|f-f_k\|_{{\rm BMO}_{\L}}+ \left(\frac{1}{|B|}\int_B |f_k(x)-(f_k)_B|^2 dx\right)^{1/2}
$$
and
$$
        \left(\frac{1}{|B|}\int_B |f(x)|dx\right)^{1/2} \lesssim \|f-f_k\|_{{\rm BMO}_{\L}}+ \left(\frac{1}{|B|}\int_B |f_k(x)|^2dx\right)^{1/2}.
$$
Hence, it follows from  $f_k\in \mathcal{B}_{\L}$ that
$ \widetilde{\gamma}_j(f)=0$    for   $ 1\leq j \leq5.$
It implies $f\in \mathcal{B}_{\L}$.  We  completed the proof of  ``(b) $\Rightarrow$ (c)''.

\bigskip

{\it Proof of ``(c) $\Rightarrow$ (a)''.}  Let $f\in \mathcal{B}_{\L}$.  We will show that  for any given $\varepsilon>0$,  there exists a function $F_\varepsilon \in C_c^\infty(\mathbb{R}^n)$ such that
\begin{equation}\label{eqn:finial-goal-varepsilon}
      \|f-F_\varepsilon \|_{{\rm BMO}_{\L}}\lesssim \varepsilon.
\end{equation}

Firstly, let $I_\varepsilon$, $J_\varepsilon$, $M_\varepsilon$, $\mathcal{R}_m$, $Q_x$ and $\mathcal{A}_{\varepsilon}(f)$ be as in the proof of  Lemma~\ref{lem:approx-type-II}.  By Remark \ref{rem:workable},  we can assume that $2^{-I_\varepsilon-1}\leq \min_{x\in \mathcal{R}_{J_\varepsilon+2}}\rho(x)$. Thus it follows from Lemma~\ref{lem:approx-type-II} that
$$
     \|f-\mathcal{A}_{\varepsilon}(f)\|_{{\rm BMO}_{\L}}\lesssim \varepsilon.
$$
This, together with \eqref{P1}, gives
\begin{align}\label{estimate-make support}
\|f-\mathcal{A}_{\varepsilon}(f)\chi_{\mathcal{R}_{M_\varepsilon+2}} \|_{{\rm BMO}_{\L}}&\leq \|f-\mathcal{A}_{\varepsilon}(f)\|_{{\rm BMO}_{\L}}+\|\mathcal{A}_{\varepsilon}(f)\chi_{\big(\mathcal{R}_{M_\varepsilon+2}\big)^c} \|_{{\rm BMO}_{\L}}\nonumber\\
&\lesssim \varepsilon+\|\mathcal{A}_{\varepsilon}(f)\chi_{\big(\mathcal{R}_{M_\varepsilon+2}\big)^c} \|_{L^\infty}\\
&\lesssim  \varepsilon. \nonumber
\end{align}
By Theorem~\ref{thm:Carleson-BMO}  again,
\begin{align*}
 &\left(\frac{1}{|B|}\int_B \Big|\mathcal{A}_{\varepsilon}(f)\chi_{\mathcal{R}_{M_\varepsilon+2}}(x)-\left(\mathcal{A}_{\varepsilon}(f)\chi_{\mathcal{R}_{M_\varepsilon+2}}\right)_B\Big|^2 dx\right)^{1/2}\\
 &\qquad\qquad\lesssim \|f-\mathcal{A}_{\varepsilon}(f)\chi_{\mathcal{R}_{M_\varepsilon+2}} \|_{{\rm BMO}_{\L}}+ \left(\frac{1}{|B|}\int_B |f(x)-\left(f\right)_B|^2 dx\right)^{1/2}
\end{align*}
and
\begin{align*}
 \left(\frac{1}{|B|}\int_B |\mathcal{A}_{\varepsilon}(f)\chi_{\mathcal{R}_{M_\varepsilon+2}}(x)|^2 dx\right)^{1/2}\lesssim \|f-\mathcal{A}_{\varepsilon}(f)\chi_{\mathcal{R}_{M_\varepsilon+2}} \|_{{\rm BMO}_{\L}}+ \left(\frac{1}{|B|}\int_B |f(x)|^2 dx\right)^{1/2}.
\end{align*}
These two estimates, together with \eqref{estimate-make support},  ensure that the  estimates \eqref{eqn:BMO1-approx-a}-\eqref{eqn:BMO1-approx-e} still hold with the parameter $R>2^{M_\varepsilon+4}$ whenever replacing $f$ by $\mathcal{A}_{\varepsilon}(f)\chi_{\mathcal{R}_{M_\varepsilon+2}}$ (the constants therein should be changed accordingly).  Recalling that $A_{t}$ is  defined in \eqref{eqn:approx-operator}, it is clear that   $A_{t}\left(\mathcal{A}_{\varepsilon}(f)\chi_{\mathcal{R}_{M_\varepsilon+2}}\right)\in C_c^\infty(\mathcal{R}_{M_\varepsilon+3})$ for any $t\leq 1$. It follows from  the proof of (ii) of Lemma~\ref{lem:CMO-identity-approx} that there exists $t_\varepsilon<<1$ sufficiently small  such that
\begin{align*}
    &\left\|\mathcal{A}_{\varepsilon}(f)\chi_{\mathcal{R}_{M_\varepsilon+2}}-A_{t_\varepsilon}\left(\mathcal{A}_{\varepsilon}(f)\chi_{\mathcal{R}_{M_\varepsilon+2}}\right)\right\|_{{\rm BMO}_{\L}}\\
    &\qquad\lesssim \varepsilon+ \sup_{B: B\subset {B(0,R)}^c}\sup_{|z|\leq 1}
   \left(\frac{1}{|\tau_z(B)|}\int_{\tau_z(B)} \left|\mathcal{A}_{\varepsilon}(f)(x)\chi_{\mathcal{R}_{M_\varepsilon+2}}(x)\right|^2dx\right)^{1/2} \lesssim \varepsilon.
\end{align*}
Therefore, we obtain  \eqref{eqn:finial-goal-varepsilon} by taking
 $F_\varepsilon=A_{t_\varepsilon}\left(\mathcal{A}_{\varepsilon}(f)\chi_{\mathcal{R}_{M_\varepsilon+2}}\right)$.  The proof of ``(c) $\Rightarrow$ (a)'' is completed.

\medskip

The implication ``(c) $\Rightarrow$ (d)'' is obvious.

\medskip

{\it Proof of  ``(d) $\Rightarrow$ (c)".}  We first show  $\widetilde{\gamma}_4(f)=0$  can be deduced by $\widetilde{\gamma}_5(f)=0$.  To this end, we use the notation in Lemma~\ref{lem:approx-type-II}, then
	\begin{align*}
		\widetilde{\gamma}_4(f)		\leq & \lim_{a>2^{M_\varepsilon+2},\,  a\to \infty} \ \sup_{Q\cap \mathcal{R}_{M_\varepsilon}=\emptyset:\, \ell(Q)\geq \rho(c_Q) } \left(\frac{1}{|Q|}\int_Q |f(x)|^2 dx \right)^{1/2}\\
		&\qquad+  \lim_{a>2^{M_\varepsilon+2},\,  a\to \infty} \ \sup_{Q\cap \mathcal{R}_{M_\varepsilon}\neq \emptyset:\, \ell(Q)\geq a} \left(\frac{1}{|Q|}\int_Q |f(x)|^2 dx \right)^{1/2}\\
		=: & I(f)+II(f).
	\end{align*}
Observe that $I(f)<\varepsilon/2$ by \eqref{eqn:CMO-e} which is a consequence of  $\widetilde{\gamma}_5(f)=0$.  Besides, for any given $Q$ involved in the term $II(f)$, it's clear that $Q\cap \mathcal{R}_{M_\varepsilon}\neq \emptyset$ and $Q\cap \left(\mathcal{R}_{M_\varepsilon +1}\right)^c \neq \emptyset$, due to $\ell(Q)>2^{M_\varepsilon+2}$ is assumed therein. Therefore, there exists a positive integer $\kappa_Q$ such that
$$
Q\subseteq \mathcal{R}_{\kappa_Q} \quad {\rm and }\quad
   |\mathcal{R}_{\kappa_Q}| \leq  2^{3n}|Q|.
$$
Hence, for such $Q$,
\begin{align*}
	\frac{1}{|Q|}\int_Q |f(x)|^2 dx&\leq  \frac{1}{|Q|}\int_{\mathcal{R}_{\kappa_Q}}  |f(x)|^2 dx\\
&\leq \frac{\|f\|_{L^2(\mathcal{R}_{M_\varepsilon})}^2}{|Q|}  +\frac{2^{3n}}{|\mathcal{R}_{\kappa_Q}|} \sum_{Q_x\subseteq \mathcal{R}_{\kappa_Q}\setminus \mathcal{R}_{M_\varepsilon}}  \int_{Q_x} |f(x)|^2 dx\leq \frac{\|f\|_{L^2(\mathcal{R}_{M_\varepsilon})}^2}{|Q|}  +2^{3n} \cdot  \Big(\frac{\varepsilon}{2}\Big)^2,
\end{align*}
where the last inequality we used
\eqref{P1}, which is  a consequence of  $\widetilde{\gamma}_5(f)=0$. Meanwhile, $\|f\|_{L^2(\mathcal{R}_{M_\varepsilon})}$ is bounded and independent of $Q$, by noticing $f\in L_{\rm loc}^2(\mathbb{R}^n)$.
Combining these estimates above, we obtain
$$
   \widetilde{\gamma}_4(f)\lesssim \varepsilon
$$
for arbitrary given $\varepsilon>0$.

\smallskip

Next, we will show that $\widetilde{\gamma}_2(f)=0$ can be deduced by  $\widetilde{\gamma}_3(f)=0$ and $\widetilde{\gamma}_5(f)=0$. Similarly to the argument in (i) above,
\begin{align*}
	\widetilde{\gamma}_2(f) \leq & \lim_{a>2^{M_\varepsilon +2},\, a\to \infty} \sup_{Q\cap \mathcal{R}_{M_\varepsilon}   =\emptyset, \, \ell(Q)\geq a}  \left(\frac{1}{|Q|}\int_Q |f(x)-f_Q|^2 dx \right)^{1/2}\\
	&\qquad +   2\lim_{a>2^{M_\varepsilon +2},\,a\to \infty} \sup_{Q\cap \mathcal{R}_{M_\varepsilon}   \neq \emptyset, \, \ell(Q)\geq a}  \left(\frac{1}{|Q|}\int_Q |f(x)|^2 dx \right)^{1/2}\\
	=:& I'(f)+2 \cdot II(f),
\end{align*}
where $II(f)$ is the second term occurred in (i) above, and we have shown that $II(f)\lesssim \varepsilon$ holds by $\widetilde{\gamma}_5(f)=0$.
Observe that $I'(f)<\varepsilon/(5\cdot 2^n)$ by \eqref{eqn:CMO-c} which is a consequence of  $\widetilde{\gamma}_3(f)=0$. Hence we obtain $\widetilde{\gamma}_2(f)=0$.

The proof of Theorem C is completed.
\end{proof}

\smallskip

\begin{remark}\label{rem:gamma4}
Assume that $\sup_{x\in \mathbb{R}^n}\rho(x)<+\infty$.  $\widetilde{\gamma}_3(f)=0$ is a consequence of $\widetilde{\gamma}_1(f)=\widetilde{\gamma}_5(f)=0$. We refer to  \cite{D} in the case of $\rho\equiv 1$.
However, in general, we could not deduce $\widetilde{\gamma}_3(f)=0$ by combining $\widetilde{\gamma}_1(f)=0$ and $\widetilde{\gamma}_5(f)=0$. In fact,  one can  construct a function $f\in {\rm BMO}_{\L}$ satisfying $\widetilde{\gamma}_1(f)=\widetilde{\gamma}_5(f)=0$, while $\widetilde{\gamma}_3(f)\neq 0$.

 To clarify this fact,  consider  the potential
 \begin{equation}\label{eqn:special-potential}
     V(x)=\frac{1}{|x|^{2-\varepsilon}},\quad {\rm  where}\quad  \varepsilon=2-(n/q_0)
 \end{equation}
  for any given $q_0>n/2$. Then $V\in {\rm RH}_q$ for any $q<q_0$. See \cite[p. 545]{Shen1}.

We first observe that
\begin{align}\label{example}
\rho(x)\approx |x|^{1-\frac{\varepsilon}{2}} \qquad {\rm for }\quad |x|>> 1.
\end{align}
  In fact, if $|x|>>1$, $y\in B(x,r)$ and $r>|x|/2$, we have $|y|\leq |x|+r\leq 3r$. Then
 \begin{align*}
      I_r(x):=\frac{1}{r^{n-2}}\int_{B(x, r)} V(y)\, d y=\frac{1}{r^{n-2}}\int_{B(x, r)} \frac{1}{|y|^{2-\varepsilon}}\, d y \geq    \frac{v_n}{3^{2-\varepsilon}} r^\varepsilon  \geq \frac{v_n}{3^{2-\varepsilon}} \left(\frac{|x|}{2}\right)^\varepsilon>1,
 \end{align*}
 where $v_n$ is the volume of the unit ball in $\mathbb{R}^n$. So, the conditions $|x|>>1$ and $I_r(x)\leq 1$ imply $r\leq |x|/2$.
Furthermore, it yields $|y|\approx |x|$  whenever $y\in B(x,r)$. One has
 $$
      I_r(x)\approx \frac{r^2}{|x|^{2-\varepsilon}},
 $$
and therefore    $I_r(x)\approx 1$ is equivalent to $ r\approx |x|^{1-\frac{\varepsilon}{2}}$.  By the definition \eqref{eqn:critical-funct},  we showed  \eqref{example}.

Next, let's choose a function $\varphi\in C_c(\mathbb{R}^n)$ satisfying ${\rm supp\,}\varphi\subseteq B(0,1)$ and $\|\varphi\|_{L^\infty}\approx \|\varphi\|_{{\rm BMO}}\approx 1$. Denote $x_k:=(3^k,0,\cdots,0)\in {\mathbb R}^n$ for $k\in \mathbb{N}$. We define
$$
    f(x)=\sum_{k=1}^\infty \varphi(x-x_k),\  x\in \mathbb{R}^n.
$$
Notice that the support sets of $\left\{\varphi(\cdot-x_k)\right\}_{k=1}^\infty$ are mutually disjoint. Then we have that  $f$ is  uniformly continuous and bounded on $\mathbb{R}^n$. So $f\in {\rm BMO}_{\L}$ and $\widetilde{\gamma}_1(f)=0$.
Besides, $\widetilde{\gamma}_3(f)\geq \|\varphi\|_{{\rm BMO}}\approx 1.$

Lastly, let us estimate $\widetilde{\gamma}_5(f)$. Suppose $a>>1$ and $B:=B(x_B,r_B)\subseteq B(0, a)^{c}$ with $r_B\geq \rho(x_B)$. Noting that $|x_B|\geq a+r_B$, it follows from \eqref{example} that $r_B\gtrsim  |x_B|^{1-\frac{\varepsilon}{2}}\geq \left(a+r_B\right)^{1-\frac{\varepsilon}{2}}$.
\begin{align}
|B|^{-1} \int_{B}\big|\sum_{k=1}^\infty \varphi(x-x_k)\big|^{2} d x&=|B|^{-1}\sum_{k: B\cap B(x_k,1)\neq \emptyset} \int_{B}\left|\varphi(x-x_k)\right|^{2} d x\nonumber\\
&\lesssim \frac{\log_3r_B}{\left(a+r_B\right)^{(1-\frac{\varepsilon}{2})n}}\lesssim a^{-\frac{n}{2}(1-\frac{\varepsilon}{2})},
\end{align}
where in the first inequality above we used one observation $\#\{k: B\cap B(x_k,1)\neq \emptyset\}\lesssim \log_3r_B$.
This gives
$$
  \widetilde{\gamma}_5(f)\lesssim \lim_{a\to \infty} a^{-\frac{n}{4}(1-\frac{\varepsilon}{2})}=0.
$$
\end{remark}

\bigskip




\section{Proof of Theorem A}\label{sec:Poisson-CMO-trace}
\setcounter{equation}{0}

With Theorem B and Theorem C at our  disposal,      we now prove   Theorem~A.  Let us  begin by introducing the following key estimates on the space derivative of the Poisson kernel of $e^{-t\sqrt{\L}}$, which were first proved  by Jiang and Li  in \cite{JL}.

\smallskip

\begin{lemma}\label{lem:space-derivative-Poisson}
{\rm (\cite[Proposition~5.2]{JL}.)} \
Let $V\in {\rm RH}_q$ for some $q>n/2$. Suppose $\int_{{\mathbb R}^n} \frac{|f(x)|}{(1+|x|)^{n+1}} \, dx<\infty$.  Then there exists a constant $C>0$ such that for any ball $B=B(x_B,r_B)$, it holds
\begin{align}\label{eqn:aux-Poisson-derivative-1}
     \int_0^{r_B} \int_B \left|t\nabla_x e^{-t\sqrt{\L}} f \right|^2 \frac{dx\, dt}{t}
    \leq C \int_0^{2r_B} \int_{2B} \bigg(\left| t^2 \partial_t ^2  e^{-t\sqrt{\L}} f\right|    \left
    |    e^{-t\sqrt{\L}}f\right|   +\frac{t^2}{r_B ^2}   \left| e^{-t\sqrt{\L}}f \right|^2\bigg)\,\frac{dx\, dt}{t}.
\end{align}
Moreover, for any constant $c_0\neq 0$, it holds

\begin{align}\label{eqn:aux-Poisson-derivative-2}
     \int_0^{r_B} \int_B \left|t\nabla_x e^{-t\sqrt{\L}} f(x) \right|^2 \frac{dx\, dt}{t}
    \leq & C \int_0^{2r_B} \int_{2B} \bigg(\left| t^2 \partial_t ^2  e^{-t\sqrt{\L}} f\right|    \left
    |    e^{-t\sqrt{\L}}f-c_0\right|   +\frac{t^2}{r_B ^2}   \left| e^{-t\sqrt{\L}}f-c_0\right|^2\bigg)\,\frac{dx\, dt}{t}\nonumber\\
    &+ C \int_0^{2r_B}\int_{2B} t \left|e^{-t\sqrt{\L} } f\right|  \left|e^{-t\sqrt{\L} } f-c_0 \right| V dx\, dt.
\end{align}
\end{lemma}

\bigskip

We now prove the main result of this article,  Theorem A.

 \begin{proof}[Proof of Theorem~A]
 (i).    If $u\in {\rm HCMO}_{\L}(\mathbb{R}_+^{n+1})$, then $u\in {\rm HMO}_{\L}(\mathbb{R}_+^{n+1})$. By Theorem~1.1 in \cite{DYZ} (or  Theorem~1.1 in \cite{JL}),  there exists a function $f\in {\rm BMO}_{\L}(\mathbb{R}^n)$ such that $u(x,t)=e^{-t\sqrt{\L}}f(x)$ and $\|f\|_{{\rm BMO}_{\L}(\mathbb{R}^n)}\leq C \|u\|_{{\rm HMO}_{\L}(\mathbb{R}_+^{n+1})}$.
  It follows from the definition of $u=e^{-t\sqrt{\L}}f\in {\rm HCMO}_{\L}(\mathbb{R}^n)$ that  $t\sqrt{\L}e^{-t\sqrt{\L}}f\in T_{2,C}^\infty$. Applying Theorem~B, we have $f\in {\rm CMO}_{\L}(\mathbb{R}^n)$ as desired.

\bigskip

(ii).
If $f\in {\rm CMO}_{\L}(\mathbb{R}^n)$,  then $f\in {\rm BMO}_{\L}(\mathbb{R}^n)$. By noting that $V\in {\rm RH}_q$ for some $q\geq (n+1)/2$, it follows from  Theorem~1.1 in \cite{JL} that
$u(x,t):=e^{-t\sqrt{\L}}f(x)\in {\rm HMO}_{\L}(\mathbb{R}_+^{n+1})$ and  $\|u\|_{{\rm HMO}_{\L}(\mathbb{R}_+^{n+1})}\leq C \|f\|_{{\rm BMO}_{\L}(\mathbb{R}^n)}$. Moreover, using Theorem~B, we know $t\partial_t u(x,t)=t\sqrt{\L}e^{-t\sqrt{\L}} f(x)\in T_{2,C}^\infty$. Thus, to prove $u\in {\rm HCMO}_{\L}$, it remains to prove
$\widetilde{\beta}_1(f)=\widetilde{\beta}_2(f)=\widetilde{\beta}_3(f)=0$, where
\begin{align*}
    \widetilde{\beta}_{1}(f)&= \lim _{a \rightarrow 0}\sup _{B: r_{B} \leq a}\left(r_B^{-n} \int_0^{r_B}\int_{B}\left|t\nabla_x e^{-t\sqrt{\L}}f(x)\right|^{2} \frac{d x \, dt}{t}\right)^{1 / 2} ,\\
    \widetilde{\beta}_{2}(f)&= \lim _{a \rightarrow \infty}\sup _{B: r_{B} \geq a}\left(r_B^{-n} \int_0^{r_B}\int_{B}\left|t\nabla_x e^{-t\sqrt{\L}}f(x)\right|^{2} \frac{d x \,dt}{t}\right)^{1 / 2},\\
   \widetilde{ \beta}_{3}(f)&= \lim _{a \rightarrow \infty}\sup _{B \subseteq \left(B(0, a)
    \right)^c}\left(r_B^{-n} \int_0^{r_B}\int_{B}\left|t\nabla_x e^{-t\sqrt{\L}}f(x)\right|^{2} \frac{d x \,dt}{t}\right)^{1 / 2}.
\end{align*}

\medskip

To this end, for any given  ball $B=B(x_B,r_B)\subseteq \mathbb{R}^n$,  split the function $f$ into three parts as follows
$$
    f=(f-f_{4B})\chi_{4B}+(f-f_{4B})\chi_{(4B)^c}+f_{4B}=:f_1+f_2+f_3,
$$
where $4B:=B(x_B,4r_B)$.  For $i=1,2,3$,  we denote
$$
J_{B,i}:=\left(r_B^{-n}\int_0^{r_B}\int_{B} \left|t \nabla_x e^{-t\sqrt{\L}} f_i(x)\right|^2  \frac{dx\, dt}{t} \right)^{1/2}.
$$
Then
$$
\left(r_B^{-n}\int_0^{r_B}\int_{B} \left|t \nabla_x e^{-t\sqrt{\L}} f(x)\right|^2  \frac{dx\, dt}{t} \right)^{1/2}\leq \sum_{i=1}^3J_{B,i}.
$$

Let us first estimate  $J_{B,1}$.  By the well known fact that the Riesz transform $\nabla_x \L^{-1/2}$ is bounded on $L^2(\mathbb{R}^n)$, one may  obtain
\begin{align}\label{eqn:aux-HCMO-1}
	J_{B,1}&=\left( r_B^{-n} \int_0^{r_B}\int_B \left| \nabla_x  \L^{-1/2}  t\sqrt{\L} e^{-t\sqrt{\L}} f_1(x)\right|^2  \frac{dx\, dt}{t}\right)^{1/2} \nonumber\\
	&\leq C \left( r_B^{-n} \int_0^{\infty}\int_{\mathbb{R}^n}  \left|  t\sqrt{\L}  e^{-t\sqrt{\L}} f_1(x)\right|^2  \frac{dx\, dt}{t} \right)^{1/2}\nonumber\\
	&\leq C\left( r_B^{-n}\int_{4B }|f-f_{4B}|^2  dx\right)^{1/2} ,
\end{align}
where  we  used \eqref{eqn:HFC-spectral} in the last inequality above.

\smallskip

Consider the second  term $J_{B,2}$.    One may apply  \eqref{eqn:aux-Poisson-derivative-1}      to obtain
\begin{align*}
    J_{B,2}&\leq C \left(r_B^{-n}   \int_0^{2r_B} \int_{2B} \bigg(\left| t^2 \partial_t ^2  e^{-t\sqrt{\L}} f_2(x)\right|    \left
    |    e^{-t\sqrt{\L}}f_2(x)\right|   +\frac{t^2}{r_B ^2}   \left| e^{-t\sqrt{\L}}f_2 (x)\right|^2\bigg)\,\frac{dx\, dt}{t} \right)^{1/2}.
\end{align*}
Then for any $x\in 2B$ and $t<2r_B$,  it follows from (i) and  (ii) of Lemma~\ref{lem:Poisson-kernels}  to see that for $m=0,2$,

\begin{align*}
	\left|   t^m \partial_t ^m e^{-t\sqrt{\L}}f_2(x)\right| &\leq C\int_{(4B)^c}   \frac{t}{|y-x_B|^{n+1}} \big|f(y)-f_{4B}\big|\, dy\\
	&\leq C \sum_{k=1}^\infty  \frac{t}{(4^k r_B)^{n+1}}  \int_{4^{k+1}B\setminus {4^k B}}  \big|f(y)-f_{4B}\big| \, dy   \\
	&\leq C \left(\frac{t}{r_B}\right)  \sum_{k=1}^\infty 4^{-k} \frac{1}{|4^{k+1}B|} \int_{4^{k+1}B}  \big|f(y)-f_{4B}\big| \, dy.
\end{align*}
Note that
\begin{align*}
\frac{1}{|4^{k+1}B|} \int_{4^{k+1}B}  \big|f(y)-f_{4B}\big| \, dy&\leq \frac{1}{|4^{k+1}B|} \int_{4^{k+1}B}  \big|f(y)-f_{4^{k+1}B}\big| \, dy+\left |f_{4B}-f_{4^{k+1}B}\right |\\
&\leq \frac{1}{|4^{k+1}B|} \int_{4^{k+1}B}  \big|f(y)-f_{4^{k+1}B}\big| \, dy+\sum_{j=1}^k \left |f_{4^{j}B}-f_{4^{j+1}B}\right |\\
&\leq  (4^nk+1) \sup_{1\leq j\leq k}\frac{1}{|4^{j+1}B|} \int_{4^{j+1}B}  \big|f(y)-f_{4^{j+1}B}\big| \, dy.
\end{align*}
Then  for any $x\in 2B$, $t<2r_B$ and $m=0,2$,
$$
    \left|  t^m \partial_t ^m  e^{-t\sqrt{\L}}f_2(x)\right|\leq C\left(\frac{t}{r_B}\right)  \sum_{k=1}^\infty 2^{-k} \sigma_k(f,B),
$$
where
$$
   \sigma_k(f,B):=\sup\limits_{0\leq j\leq k}\frac{1}{|4^{j+1}B|} \int_{4^{j+1}B}  \big|f(y)-f_{4^{j+1}B}\big| \, dy.
$$
This gives
\begin{equation}\label{eqn:aux-HCMO-2}
	J_{B,2}\leq C \sum_{k=1}^\infty 2^{-k}\, \sigma_{k}(f,B).
\end{equation}

\smallskip

Consider the term $J_{B,3}$.
 By applying (i) and  (iii) of Lemma~\ref{lem:Poisson-kernels}, we have
\begin{equation}\label{eqn:aux-Poisson-f3}
     \left| e^{-t\sqrt{\L}} (f_{4B})(x) \right|\leq C |f_{4B}|   \quad {\rm and }\quad
     \left|t^2 \partial_t^2 e^{-t\sqrt{\L}} (f_{4B})(x)\right| \leq C  \left(\frac{t}{\rho(x)}\right)^\delta  \left(1+\frac{t}{\rho(x)}\right)^{-N} |f_{4B}|
\end{equation}
for each $(x,t)\in \mathbb{R}_+^{n+1}$,  where $\delta>0$ is the parameter in  Lemma~\ref{lem:Poisson-kernels}.
We consider  the following two cases.
\smallskip

\noindent{\it Case 1.} $2r_B\geq \rho(x_B)$. In this case, it follows from Corollary~1 in \cite{DGMTZ} that we can select a finite family of critical balls $\left\{B\big(x_i, \rho(x_i)\big)\right\}$ such that
$$
    2B\subseteq \bigcup_i B(x_i, \rho(x_i))\quad {\rm and }\quad \sum_i \left|B\big(x_i, \rho(x_i)\big)\right|\leq c |B|,
$$
where  $c=c(\rho)<\infty$ independent of $B$. Hence, $\rho(x)\approx \rho(x_i)$ for each $x\in B(x_i, \rho(x_i))$, and it follows from  \eqref{eqn:aux-Poisson-derivative-1}  and \eqref{eqn:aux-Poisson-f3}  to see
\begin{align*}
	(J_{B,3})^2 &\leq C   r_B^{-n}\sum_i    \int_0^{2r_B}\int_{B(x_i, \rho(x_i))}  \left[   \left(\frac{t}{\rho(x)}\right)^\delta  \left(1+\frac{t}{\rho(x)}\right)^{-N} |f_{4B}|^2   +   \frac{t^2}{r_B^2 }  |f_{4B}|^2\right] \frac{dx\, dt}{t}\\
	&\leq C |f_{4B}|^2  r_B^{-n}   \sum_i   \left[  \int_0^{\rho(x_i)}\int_{B(x_i, \rho(x_i))}   \left(\frac{t}{\rho(x_i)}\right)^\delta  \frac{dx\, dt}{t}\right.\\
&\hskip 3cm \left.+\int_{\rho(x_i)} ^\infty \int_{B(x_i, \rho(x_i))}  \left(\frac{t}{\rho(x_i)}  \right)^{-N}	 \frac{dx\, dt}{t}  +\left|B\big(x_i, \rho(x_i)\big)\right|\right]\\
	&\leq C |f_{4B}|^2 r_B^{-n} \sum_i \left|B\big(x_i, \rho(x_i)\big)\right|\ \leq C  |f_{4B}|^2.
\end{align*}

\smallskip

\noindent{\it Case 2.} $2r_B< \rho(x_B)$.  For any $x\in 2B$, we get $|x-x_B|<2r_B<\rho(x_B)$.  It follows from  Lemma ~\ref{lem:size-rho}  that $\rho(x)\approx \rho(x_B)$, for any $x\in 2B$.  We then apply   \eqref{eqn:aux-Poisson-derivative-2} by taking the parameter $c_0:=f_{4B}$, and  combine \eqref{eqn:aux-Poisson-f3} and  \eqref{eqn:aux-compare-Poisson} to obtain
\begin{align*}
    \left(J_{B,3}\right)^2 \leq\, & C r_B^{-n}\int_0^{2r_B}\int_{2B}  \left[ |f_{4B}|^2 \left(\frac{t}{\rho(x)}\right)^{\delta+2-n/q} +\frac{t^2}{r_B^2}  |f_{4B}|^2 \left(\frac{t}{\rho(x)}\right)^{2 (2-n/q)}\right]\frac{dx\, dt}{t}\\
    &+ Cr_B^{-n}\int_0^{2r_B}\int_{2B} t |f_{4B}|^2   \left(\frac{t}{\rho(x)}\right)^{2-n/q}  V(x) dx\, dt.
    \end{align*}
By Lemma \ref{V-integral size}, we have the fact  $\int_{2B} V(x)\, dx \leq C (2r_B)^{n-2}$ for $2r_B< \rho(x_B)$. Then,   one can get
$$
 \left(J_{B,3}\right)^2 \leq  C  |f_{4B}|^2  \left(\frac{r_B}{\rho(x_B)}\right)^{\delta}
$$
since   $0<\delta<2-n/q$.

Combining {\it Case 1} and {\it Case 2} above, we obtain
\begin{equation}\label{eqn:aux-HCMO-3}
	J_{B,3} \leq C \, |f_{4B}| \min\bigg\{  \bigg(\frac{2r_B}{\rho(x_B)}\bigg)^{\delta/2}, 1\bigg\}.
\end{equation}

\medskip

We are now in a position to prove the aimed $\widetilde{\beta}_1(f)=\widetilde{\beta}_2(f)=\widetilde{\beta}_3(f)=0$.   By  \eqref{eqn:aux-HCMO-1} and \eqref{eqn:aux-HCMO-2},
 \begin{align}\label{estimate:J1 and J2}
    J_{B,1}+J_{B,2}
  \leq C \sum_{k=1}^\infty 2^{-k}  \sigma_{k}(f,B).
\end{align}	
Since $f\in {\rm CMO}_{\L}$,  it follows  from Theorem~C that  $\widetilde{\gamma}_j(f)=0,  j=1,\dots, 5$.  Then one can show that for any $k\in \mathbb{N}$
\begin{equation}\label{eqn:aux-sigma-k-1}
    \lim_{a\to 0}\sup_{B:\, r_B\leq a} \sigma_{k}(f,B)=\lim_{a\to \infty}\sup_{B:\, r_B\geq a} \sigma_{k}(f,B) =\lim_{a\to \infty} \sup_{B: B\subseteq (B(0,a))^c} \sigma_{k}(f,B)=0.
\end{equation}
 In fact, the first two terms  in \eqref{eqn:aux-sigma-k-1}   vanish due to $\widetilde{\gamma}_1(f)=0$ and $\widetilde{\gamma}_2(f)=0$, respectively. To estimate the third term in \eqref{eqn:aux-sigma-k-1}, for any given large positive number $a$, we  classify balls $B\subseteq (B(0,a))^c$  by  the size of $B$. Then we can use $\widetilde{\gamma}_2(f)=0$ ( in the case of  $r(B)\geq R_0$,  where $R_0>>1$) and  $\widetilde{\gamma}_3(f)=0$ (in the case of $r_B< R_0$) to prove $\lim\limits_{a\to \infty} \sup\limits_{B\subseteq (B(0,a))^c} \sigma_{k}(f,B)=0$.

 For any $\varepsilon>0$, there exists $N>0$, such that $\sum_{k=N}^\infty 2^{-k}<\varepsilon$.  By noting that  $\sigma_{k}(f,B)\leq\|f\|_{\rm BMO}\leq \|f\|_{{\rm BMO}_{\L}}$ for any $k\in \mathbb{N}$, we then have
$$
\sum_{k=1}^\infty 2^{-k}  \sigma_{k}(f,B)\leq \sum_{k=1}^N 2^{-k}  \sigma_{k}(f,B)+\sum_{k=N+1}^\infty 2^{-k} \|f\|_{{\rm BMO}_{\L}}\leq  \sum_{k=1}^N 2^{-k}  \sigma_{k}(f,B)+\varepsilon\|f\|_{{\rm BMO}_{\L}},$$
which, together with \eqref{eqn:aux-sigma-k-1}, gives
\begin{align}
  \lim_{a\to 0}\sup_{B:\, r_B\leq a} \big(J_{B,1}+J_{B_2}\big)=\lim_{a\to \infty}\sup_{B:\, r_B\geq a} \big(J_{B,1}+J_{B_2} \big)=\lim_{a\to \infty} \sup_{B\subseteq (B(0,a))^c} \big(J_{B,1}+J_{B_2} \big)=0.
\end{align}

\smallskip

In the end, we are concerned with the behavior of  $J_{B,3}$ as $B$ is small, or large, or far away from the origin.  Note  that when $r_B<\rho(x_B)/4$,   one can apply Lemma \ref{average on B} to obtain
\begin{align}\label{constant term}
   |f_{4B}|&\leq C\min\bigg\{\,|f|_{B(x_B,\rho(x_B))} \bigg(\frac{\rho(x_B)}{r_B}\bigg)^{n},\,  \|f\|_{{\rm BMO}_{\L}} \bigg(1+\log \frac{\rho(x_B)}{r_B} \bigg)\,\bigg\}\nonumber\\
   &\leq C \left(|f|_{B(x_B,\rho(x_B))} \bigg(\frac{\rho(x_B)}{r_B}\bigg)^{n}\right)^{1-\theta} \left(\|f\|_{{\rm BMO}_{\L}} \left(1+\log \frac{\rho(x_B)}{r_B}\right)\right)^\theta,
\end{align}
for any $\theta\in [0,1]$.
It follows from  \eqref{eqn:aux-HCMO-3} and \eqref{constant term} that
\begin{align*}
	\lim_{a\to \infty}\sup_{B:\, r_B\geq  a,\, r_B< \rho(x_B)/4} J_{B,3}&\lesssim  \lim_{a\to \infty}\sup_{B:\, r_B\geq  a,\, r_B< \rho(x_B)/4} |f_{4B}| \,\bigg(\frac{r_B}{\rho(x_B)}\bigg)^{\delta/2}\\
	&\lesssim \lim_{a\to \infty}\sup_{B:\, r_B\geq  a,\, r_B< \rho(x_B)/4}  |f|_{B(x_B,\rho(x_B))} ^{1-\theta} \bigg(\frac{\rho(x_B)}{r_B}\bigg)^{(n-\delta/2)(1-\theta)-\delta\theta/4}  \|f\|_{{\rm BMO}_{\L}}^\theta
\end{align*}
for any $\theta\in (0,1)$. Take $\theta^*\in (0,1)$ sufficiently close to 1,  such that $(n-\delta/2)(1-\theta^*)-\delta\theta^*/4<0$. Then it follows from $\widetilde{\gamma}_4(f)=0$ that
$$
	\lim_{a\to \infty}\sup_{B:\, r_B\geq  a \, r_B< \rho(x_B)/4} J_{B,3}\lesssim \lim_{a\to \infty}\sup_{B:\, r_B\geq  a,\, r_B< \rho(x_B)/4}  |f|_{B(x_B,\rho(x_B))} ^{1-\theta^*} \|f\|_{{\rm BMO}_{\L}}^{\theta^*} \lesssim \|f\|_{{\rm BMO}_{\L}}^{\theta^*} \left(\widetilde{\gamma}_4(f)\right)^{1-\theta^*}=0,
$$
which, together with  the fact $ \displaystyle \lim_{a\to \infty}\sup_{B:\, r_B\geq \max\{a, \rho(x_B)/4\}}  |f_{4B}| \lesssim \widetilde{\gamma}_4(f)$, implies that
$$
\lim_{a\to \infty}\sup_{B:\, r_B\geq  a} J_{B,3}=0.
$$

Similarly,  one may apply $\widetilde{\gamma}_4(f)=\widetilde{\gamma}_5(f)=0$ to obtain
$$
\lim_{a\to \infty} \sup_{B:B\subseteq (B(0,a))^c,\,r_B\geq \rho(x_B)/4}J_{B,3} \lesssim \lim_{a\to \infty} \sup_{B:B\subseteq (B(0,a))^c,\, r_B\geq \rho(x_B)/4} |f_{4B}|\lesssim \widetilde{\gamma}_5(f)=0
$$
and
\begin{align*}
&\lim_{a\to \infty} \sup_{B:B\subseteq (B(0,a))^c,\, r_B<\rho(x_B)/4} J_{B,3}\lesssim \|f\|_{{\rm BMO}_{\L}}^{\theta^*}
\lim_{a\to \infty} \sup_{B:B\subseteq (B(0,a))^c,\, r_B< \rho(x_B)/4}  |f|_{B(x_B,\rho(x_B))} ^{1-\theta^*}   \\
&\lesssim \|f\|_{{\rm BMO}_{\L}}^{\theta^*}  \lim_{a\to \infty} \left\{\sup_{B:B\subseteq (B(0,a))^c,\, r_B< \rho(x_B)/4<a/8}  |f|_{B(x_B,\rho(x_B))} ^{1-\theta^*}+\sup_{B\subseteq (B(0,a))^c:\, \max\{r_B,a/8\}< \rho(x_B)/4}  |f|_{B(x_B,\rho(x_B))} ^{1-\theta^*}\right\}\\
&\lesssim \|f\|_{{\rm BMO}_{\L}}^{\theta^*}\left\{ \left( \widetilde{\gamma}_5(f)\right)^{1-\theta^*} + \left( \widetilde{\gamma}_4(f)\right)^{1-\theta^*} \right\}=0.
\end{align*}
These give that
\begin{align}\label{B is far}
\lim_{a\to \infty} \sup_{B:B\subseteq (B(0,a))^c}J_{B,3}=0.
\end{align}

\smallskip

 It remains to prove $\displaystyle \lim_{a\to 0}\sup_{B:\, r_B\leq a} J_{B,3}=0$.  Due to \eqref{B is far}, for any $\varepsilon>0$, there exists $R_\varepsilon>>1$ such that $\displaystyle \sup_{B\subseteq B(0,R_\varepsilon)^c}J_{B,3}<\varepsilon$. One can write
 \begin{align}\label{last-JB3}
 	 \lim_{a\to 0}\sup_{B:\, r_B\leq a} J_{B,3}&\leq  \lim_{a\to 0}\sup_{B: B\subseteq B(0,R_\varepsilon+1),\, r_B\leq a} J_{B,3}+  \lim_{a\to 0}\sup_{B: B\subseteq B(0,R_\varepsilon)^c}J_{B,3}\nonumber\\
 &\leq \lim_{a\to 0}\sup_{B:\, B\subseteq B(0,R_\varepsilon+1),\, r_B\leq a} J_{B,3}+\varepsilon.
 \end{align}

  It follows from Lemma~\ref{lem:size-rho} that  $\inf_{x\in B(0,R_\varepsilon+1)} \rho(x)>0$ (we denote the infimum by $m_\varepsilon$). So, if $a<m_\varepsilon/4$, one can apply \eqref{eqn:aux-HCMO-3} and  Lemma \ref{average on B} to get
 \begin{align*}
 	\sup_{B: B\subseteq B(0,R_\varepsilon +1),\, r_B\leq a} J_{B,3}&\leq  C \sup_{B: B\subseteq B(0,R_\varepsilon +1),\, r_B\leq a} |f_{4B}|   \bigg(\frac{r_B}{\rho(x_B)}\bigg)^{\delta/2}\\
 	&\leq C\sup_{B: B\subseteq B(0,R_\varepsilon +1),\, r_B\leq a} \left(1+\log \frac{\rho(x_B)}{4r_B}\right)\|f\|_{{\rm BMO}_{\L}} \bigg(\frac{r_B}{\rho(x_B)}\bigg)^{\delta/2}\\
 	&\leq C \,\|f\|_{{\rm BMO}_{\L}}\frac{a^{\delta/4}}{m_\varepsilon^{\delta/4}} \to 0,  \qquad\quad {\rm as}  \quad a\to 0.
 	 \end{align*}
This, together with \eqref{last-JB3}, implies
$\displaystyle  \lim_{a\to 0}\sup_{B:\, r_B\leq a} J_{B,3}=0$,
as desired.  We finish the proof of (ii)  of Theorem~A.
\end{proof}

\bigskip


 \noindent{\bf Acknowledgments.}
 The authors would like to thank Lixin Yan for helpful suggestions. L.~Song is supported by  NNSF of China (No.~12071490). L.C.~Wu is supported by  Guangdong Basic and Applied Basic Research Foundation (No.~2019A1515110251) and Fundamental Research Funds for the Central Universities (No.~20lgpy141).

 \bigskip


\end{document}